\documentclass[a4paper, 12pt]{amsart}
\usepackage{latexsym, amssymb, amscd, amsfonts, amsthm, amsmath, graphicx}
\usepackage{url}
\usepackage[margin=1in]{geometry}  
\usepackage{epsfig}
\usepackage{color}                 
\usepackage{epsfig}
\usepackage{mathrsfs}

\newcommand{\GCD}{\operatorname{GCD}}

\newcommand{\bC}{\mathbb{C}}

\newcommand{\bN}{\mathbb{N}}
\newcommand{\bQ}{\mathbb{Q}}
\newcommand{\bR}{\mathbb{R}}
\newcommand{\bS}{\mathbb{S}}
\newcommand{\bT}{\mathbb{T}}
\newcommand{\bU}{\mathbb{U}}

\newcommand{\zed}{\mathbb{Z}}

\newcommand{\sgn}{\mathrm{sgn}}

\newcommand{\rNW}{\mathrm{NW}}
\newcommand{\rKL}{\mathrm{Kl}}

\newcommand{\E}{\mathbf{E}}

\newcommand{\one}{\mathbf{1}}

\newcommand{\sC}{{\mathscr{C}}}

\newcommand{\sF}{{\mathscr{F}}}

\newcommand{\sL}{{\mathscr{L}}}
\newcommand{\sM}{{\mathscr{M}}}

\newcommand{\sP}{{\mathscr{P}}}

\newcommand{\sX}{{\mathscr{X}}}

\newtheorem*{theorem*}{Theorem}

\newtheorem{theorem}{Theorem}
\newtheorem{cor}[theorem]{Corollary}
\newtheorem{lemma}[theorem]{Lemma}
\newtheorem{proposition}[theorem]{Proposition}

\theoremstyle{remark}
\newtheorem{rem}{Remark}

\title{The angle of large values of $L$-functions}
\author{Bob Hough}
\address{Department of Mathematics, Stanford University, 450 Serra Mall, 
Building 380, Stanford, CA, 94305}
\curraddr{Institute of Advanced Study, 1 Einstein Drive, Princeton, NJ 08540}
\email{hough@math.ias.edu}
\thanks{Research supported by ERC Grant 279438: Approximate algebraic structure 
and applications.}

\thanks{This project was conducted while the author was a 
postdoctoral research fellow at DPMMS Cambridge and then at Maths Institute, 
Oxford, and as a visitor to the Department of Math at Stanford. He thanks the 
three departments for their hospitality. }

\thanks{The author thanks the referee for careful comments.}


\begin{document}

\begin{abstract}
We prove three results on the argument of large central values of $L$-functions. 
 The first establishes that there exists a sequence of quadratic Dirichlet 
characters $\chi_d$ and Dirichlet polynomials $T(\chi_d)$ truncating 
$L(\frac{1}{2}, \chi_d)$ at a length a power of $d$, such that the truncated sum 
is large and negative.  On the generalized Riemann Hypothesis this distinguishes 
the central point $\frac{1}{2}$ from fixed $\sigma > \frac{1}{2}$.  A result of 
Kalpokas, Korolev and Steuding establishes large values of the Riemann zeta 
function among $\{\zeta(\frac{1}{2} + it): t \in [T, 2T]\}$ with prescribed 
argument modulo $\pi$, with a weaker result modulo $2\pi$.  Our second result 
removes the condition modulo $\pi$.  Our third result proves an analogue in the 
family of central values of Dirichlet $L$-functions to fixed prime conductor. 
\end{abstract}

\maketitle

\section{Introduction}
The best known method for establishing extreme values of zeta, $L$-functions and 
sums of other arithmetic harmonics is the resonance method due to Soundararajan 
\cite{S08}.  
The resonance method is a first moment method which uses an auxiliary Dirichlet 
polynomial as an indicator for the large values of the sum of interest. For 
instance,  \cite{S08} gives the omega result\footnote{$A \gtrsim B$ 
means $\liminf \frac{A}{B} \geq 1$.}
\begin{equation}\label{resonance_omega}
 \sup_{t \in [T, 2T]} \log \left|\zeta\left(\frac{1}{2}+it\right)\right| \gtrsim 
\sqrt{\frac{\log T}{\log \log T}} , \qquad T \to \infty
\end{equation}
by optimizing
\begin{equation}\label{sound}
 \int_{T}^{2T} \zeta\left(\frac{1}{2} + it\right) \left|R\left(\frac{1}{2} + 
it\right)\right|^2 dt \bigg/ \int_T^{2T} \left|R\left(\frac{1}{2} + 
it\right)\right|^2 dt 
\end{equation}
over Dirichlet polynomials
\begin{equation}
 R(s) = \sum_{n \leq T^{1-\epsilon}} \frac{r(n)}{n^s}.
\end{equation}  
We further develop the resonance method in three families of $L$-functions by 
using it to prove omega results of the type (\ref{resonance_omega}) with the 
condition that the angle of the object is  constrained. Since the completion of 
this work, the result (\ref{resonance_omega}) has been improved in \cite{BS15}, 
with the same method, but taking a longer resonating Dirichlet polynomial.  We 
are presenting our results in any case, as we expect that they may still be of 
interest. 

  Let $d \geq 0$ be the fundamental discriminant associated to real quadratic 
field $\bQ(\sqrt{d})$, with associated real Dirichlet character $\chi_d$.  
Initially defined for $\Re(s) > 1$, the Dirichlet $L$-function 
\begin{equation}L(s,\chi_d) = 
\sum_{n \geq 1} \frac{\chi_d(n)}{n^s}\end{equation} 
extends holomorphically to $\bC$ and satisfies the functional equation 
\begin{equation}
\left(\frac{d}{\pi}\right)^{\frac{s}{2}} 
\Gamma\left(\frac{s}{2}\right)L(s,\chi_d) = \Lambda(s,\chi_d) =  \Lambda(1-s, 
\chi_d). 
\end{equation}
Within the critical strip $0 \leq \Re(s) \leq 1$ the approximate functional equation of analytic number theory then permits the representation of $L(s,\chi_d)$ as the sum of two Dirichlet polynomials, the product of whose length is roughly $d$.  As a special case
\begin{equation}\label{AFE_real_char}
 L\left(\frac{1}{2}, \chi_d\right) = 2\sum_{n \geq 1} \frac{\chi_d(n)}{n^{\frac{1}{2}}} V\left(\sqrt{\frac{\pi}{d}} n\right)
\end{equation}
where $V(x)$ is an appropriate smooth function on $\bR_{\geq 0}$ satisfying 
$V(0) = 1$, e.g. 
\begin{equation}
 V(x) = \frac{1}{2\pi i} \int_{\Re s = 1}x^{-s} \Gamma\left(\frac{s}{2} \right) 
 \frac{ds}{s - \frac{1}{2}}.
\end{equation}

The Generalized Riemann Hypothesis implies that $L(\sigma, \chi_d) \geq 0$ for 
$\sigma \geq \frac{1}{2}$, so that on GRH (\ref{AFE_real_char}) is non-negative. 
 Were we to replace the ratio $\frac{\chi_d(n)}{\sqrt{n}}$ with 
$\frac{\chi_d(n)}{n^\sigma}$ for any fixed $\sigma > \frac{1}{2}$ then again the 
GRH implies asymptotic positivity of shorter smoothed Dirichlet polynomials of 
length any power of $d$.   Thus we expect a mild bias towards $\chi_d(n) = 1$. 
Our first result proves a limitation to this effect.
\begin{theorem}\label{negative_truncation_theorem}
 Fix $\phi: \bR^+ \to [0,1]$ a smooth function satisfying $\phi \equiv 1$ in a 
neighborhood of 0 and $\phi$ is supported in $[0,1]$.  Let $0 < \delta < 
\frac{2}{9}$.   Set 
 \begin{equation}
  \eta = 0.36845 \min\left( \frac{\delta}{2}, \frac{1}{18} - 
\frac{\delta}{4}\right).
 \end{equation}
For large $D$ there exists fundamental discriminant $d \asymp D$ such that
 \begin{equation}
  \sum_n \frac{\chi_d(n)}{\sqrt{n}} 
\phi\left(\frac{n}{D^\delta}\right) < - \exp\left(\sqrt{\frac{\eta\log D}{\log 
\log D}}\right).
 \end{equation}

\end{theorem} 
\begin{rem}
 Naively one might expect a result of this nature to hold for any fixed $\delta 
< \frac{1}{2}$.  Deciding the behavior for $\delta = \frac{1}{2} + o(1)$ is an 
interesting question even experimentally.
\end{rem}

On GRH, Theorem \ref{negative_truncation_theorem} thus distinguishes between 
Dirichlet polynomial truncations to $L(\frac{1}{2}, \chi_d)$ of length 
$d^{\frac{1}{2}}$ and of a smaller power of $d$, and again distinguishes the 
shorter Dirichlet polynomial truncations with those of the same length at points 
to the right of $\frac{1}{2}$.

 Our proof uses a decomposition of the resonating Dirichlet polynomial as a 
convolution of multiplicative functions supported at small and large primes. See 
\cite{GS07} for somewhat related argument which produces negative truncations to 
the Dirichlet series for $L(1,\chi_d)$ of length on the scale of $\log d$.

Our second result concerns the argument of large values of $\zeta(\frac{1}{2} + it)$.  Combining the resonance method with a contour method of averaging $\zeta(\frac{1}{2} + it)$ over generalized Gram points where it has prescribed argument modulo $\pi$, Kalpokas, Korolev and Steuding \cite{KKS13} show that for any $\theta \in \bR/\zed$, for large $T$ there exists $t \asymp T$ with 
\begin{equation}
 \frac{1}{\pi}\arg\left(\zeta\left(\frac{1}{2} + it\right)\right) \equiv  \theta \bmod \zed, \qquad \Re \log \zeta\left(\frac{1}{2} + it\right) \gtrsim  \sqrt{\frac{\log T}{\log \log T}}.
\end{equation}
This method gives large values of $\zeta$ but with angle prescribed only modulo $\pi$.   By a somewhat different argument they remove this defect, but obtain only values of size $|\zeta(\frac{1}{2} +  it)| \gg (\log T)^{\frac{9}{4}}$.  Modifying their two methods we prove the following result.
\begin{theorem}\label{large_zeta_theorem}
 For all $\theta \in \bR/\zed$, for $T$ sufficiently large there is $t \asymp T$ satisfying
 \begin{equation}
  \frac{1}{2\pi}\arg\left(\zeta\left(\frac{1}{2} + it\right)\right) \equiv \theta \bmod \zed, \quad \Re \log \zeta\left( \frac{1}{2} + it\right) \gtrsim  \sqrt{\frac{\log T}{\log \log T}}.
 \end{equation}
\end{theorem}
\noindent
In this case, the proof goes by comparing the signed and unsigned first moments 
of $\zeta$ amplified by the resonator,  averaged over  generalized Gram points.

As in the case of $\zeta$, twice the argument of a primitive Dirichlet $L$-function at the central point is well understood.  Let $q \geq 3$ be a prime and let $\chi \bmod q$ be a non-principal character, with associated $L$-function
\begin{equation}L(s,\chi)= \sum_n \frac{\chi(n)}{n^s}, \qquad \Re(s) > 
1.\end{equation}  The Gauss sum associated to $\chi$ is
\begin{equation}
 \tau(\chi) = \sum_{a \bmod q} \chi(a) e\left(\frac{a}{q}\right),
\end{equation}
and the root number of the $L$-function  is \begin{equation} \epsilon_\chi = 
\frac{\tau(\chi)}{i^a\sqrt{q}}, \qquad a = \frac{1- \chi(-1)}{2},\end{equation} 
which is a complex number of modulus 1.  The completed $L$-function is 
\begin{equation}
 \Lambda(s,\chi) =\left(\frac{q}{\pi}\right)^{\frac{s}{2}} \Gamma\left(\frac{s + a}{2}\right) L(s,\chi), 
\end{equation}
which satisfies the function equation
\begin{equation}
\Lambda(s,\chi) =   \epsilon_\chi\Lambda(1-s, \overline{\chi}). 
\end{equation}
Define $\theta_\chi \in \bR/\zed$ by $e(\theta_\chi) |L(\frac{1}{2},\chi)| = L(\frac{1}{2}, \chi)$.  Thus $\frac{\tau(\chi)}{i^a\sqrt{q}} = e(2\theta_\chi)$. Katz \cite{K88} proves that the angles $\{\frac{\tau(\chi)}{\sqrt{q}}\}_{\chi \bmod q}$ are asymptotically equidistributed with respect to Haar measure on $\bS^1 = \{z \in \bC: |z| = 1\}$ in the limit as $q\to\infty$.  Our last theorem can be seen as an extension of Katz' result.
\begin{theorem}\label{dirichlet_L_angle_theorem}
    Let $F(x)$ be a growth function, satisfying  for all large $x$, $F(x) = o\left((\log x)^{\frac{1}{2}}\right)$.  For all primes $q > q_0( F)$, for all $\theta \in \bR/\zed$, for all $\delta  > \frac{1}{F(q)}$, there exists non-principal $\chi\bmod q$ such that,
 \begin{equation}
  \left\|\frac{1}{2\pi}\arg L\left(\frac{1}{2},\chi\right) - \theta \right\|_{\bR/\zed} \leq \delta, \qquad  \Re \log L\left(\frac{1}{2},\chi\right) \gtrsim \sqrt{\frac{ \log q}{32\log \log q}}.
 \end{equation}
\end{theorem}

In this case we use a one sided variant of Weyl's criterion to prove the uniform distribution.

Our proof of Theorem \ref{dirichlet_L_angle_theorem} uses an  asymptotic for the 
twisted fourth moment of Dirichlet $L$-functions.  This asymptotic, which 
extends the  fourth moment with power saving error term given in \cite{Y11}, is 
a concurrent result of the author that will appear elsewhere.
\begin{theorem}[Twisted fourth moment]\label{twisted_fourth_moment_theorem} Let $0 \leq \vartheta < \frac{1}{32}$ and let $1 \leq \ell_1, \ell_2 \leq q^{\vartheta}$ be square-free and satisfy $(\ell_1, \ell_2) = 1$.  Given $\alpha, \beta, \gamma, \delta \in \bC$, define for square-free $\ell$ the generalized divisor function
\begin{equation}
 \tau_{\alpha, \beta, \gamma, \delta}(\ell) = \prod_{p|\ell}\left(1 + \frac{p^{\gamma-\delta} \zeta_p(2 + \alpha + \beta + \gamma + \delta)}{\zeta_p(1 +\alpha + \gamma)\zeta_p(1 + \beta + \gamma)}\right)
\end{equation}
where $\zeta_p(s) = (1-p^{-s})^{-1}$. 
Denote
\begin{equation}
 X_u = \left(\frac{q}{\pi}\right)^{-u}\frac{\Gamma\left(\frac{\frac{1}{2}-u}{2}\right)}{\Gamma\left(\frac{\frac{1}{2}+u}{2}\right)}
\end{equation}
and, for $S$ a set of parameters, $X_S = \prod_{u \in S} X_u$. Write $\E_{\chi 
\bmod q}^+$ for expectation with respect to the even non-principal characters 
modulo $q$.

There exists $\eta > 0$ such that if $\alpha, \beta, \gamma, \delta \in \left\{z 
\in \bC: |z| < \frac{\eta}{\log q}\right\}$ then, for any $\epsilon > 0$, 
\begin{align}M(\alpha, \beta, \gamma, \delta; \ell_1,\ell_2):= 
\qquad\qquad&\\ \notag \E_{\chi\bmod q}^+ 
\Biggl[\chi(\ell_1)\overline{\chi}(\ell_2) 
L\left(\frac{1}{2} + \alpha, \chi \right)&L\left(\frac{1}{2} + \beta, \chi 
\right)L\left(\frac{1}{2} + \gamma, \overline{\chi} \right)L\left(\frac{1}{2} + 
\delta, \overline{\chi} \right) \Biggr]
=\\ \notag
 \frac{\tau_{ \alpha, \beta, \gamma, \delta}(\ell_1)\tau_{ \gamma, \delta, 
\alpha, \beta}(\ell_2)}{\ell_1^{\frac{1}{2} + \gamma} 
\ell_2^{\frac{1}{2}+\alpha}}&\frac{\zeta(1 +\alpha + \gamma)\zeta(1 + \alpha + 
\delta) \zeta(1 + \beta + \gamma) \zeta(1 + \beta + \delta)}{\zeta(2 +\alpha + 
\beta +\gamma+\delta)}&\\ \notag
 + X_{\alpha,\gamma}\frac{\tau_{-\gamma, \beta, -\alpha, \delta}(\ell_1) \tau_{-\alpha, \delta, -\gamma, \beta}(\ell_2)}{\ell_1^{\frac{1}{2} -\alpha} \ell_2^{\frac{1}{2} -\gamma}} 
 &\frac{\zeta(1-\alpha -\gamma) \zeta(1-\gamma +\delta) \zeta(1 -\alpha + \beta) 
\zeta(1+\beta + \delta)}{\zeta(2-\alpha + \beta -\gamma +\delta)}\\ \notag
 + X_{\beta,\gamma}\frac{\tau_{\alpha, -\gamma, -\beta, \delta}(\ell_1) \tau_{-\beta, \delta, \alpha, -\gamma}(\ell_2)}{\ell_1^{\frac{1}{2} -\beta} \ell_2^{\frac{1}{2} + \alpha}} 
 &\frac{\zeta(1+\alpha - \beta) \zeta(1+\alpha +\delta) \zeta(1-\beta-\gamma) 
\zeta(1-\gamma + \delta)}{\zeta(2+\alpha - \beta -\gamma +\delta)}\\ \notag
  + X_{\alpha,\delta}\frac{\tau_{-\delta, \beta, \gamma, -\alpha}(\ell_1) \tau_{\gamma, -\alpha, -\delta, \beta}(\ell_2)}{\ell_1^{\frac{1}{2} +\gamma} \ell_2^{\frac{1}{2} - \delta}}  
  &\frac{\zeta(1+\gamma-\delta) \zeta(1-\alpha -\delta) \zeta(1 + \beta + 
\gamma) \zeta(1-\alpha+\beta)}{\zeta(2-\alpha + \beta +\gamma -\delta)}\\ \notag
   + X_{\beta,\delta}\frac{\tau_{\alpha, -\delta, \gamma, -\beta}(\ell_1) \tau_{\gamma, -\beta, \alpha, -\delta}(\ell_2)}{\ell_1^{\frac{1}{2} +\gamma} \ell_2^{\frac{1}{2} + \alpha}} 
   &\frac{\zeta(1+\alpha +\gamma) \zeta(1+\alpha-\beta) \zeta(1 +\gamma - 
\delta) \zeta(1-\beta - \delta)}{\zeta(2+\alpha - \beta +\gamma 
-\delta)}\\\notag
     +  X_{\alpha, \beta, \gamma, \delta}\qquad\qquad\qquad\qquad\qquad 
&\\\notag \times \frac{\tau_{ -\gamma, -\delta, -\alpha, -\beta}(\ell_1)\tau_{ 
-\alpha, -\beta, -\gamma, -\delta }(\ell_2)}{\ell_1^{\frac{1}{2} -\alpha} 
\ell_2^{\frac{1}{2}-\gamma}}
     &\frac{\zeta(1 -\alpha - \gamma)\zeta(1 - \beta - \gamma) \zeta(1 - \alpha 
- \delta) \zeta(1 - \beta - \delta)}{\zeta(2 -\alpha - \beta -\gamma-\delta)}\\ 
\notag
  &+ 
O_\epsilon\left(\frac{\max(\ell_1,\ell_2)^{\frac{1}{2}}}{q^{\frac{1}{32}
-\epsilon
} \min(\ell_1,\ell_2)^{\frac{1}{2}}}\right).
\end{align}
\end{theorem}

\subsubsection*{Acknowledgments} The author thanks K. Soundararajan,  and D.R. Heath-Brown for their encouragement.

 \section*{Notation and conventions}
The analytic conductor of a family of $L$-functions refers to the quantity $D$ for fundamental discriminants, $T$ for $\zeta$ or $q$ for Dirichlet $L$-functions, determining the family of harmonics.
We consider parameters taken in the limit of growing conductor.  Given two such positive parameters,  $A \sim B$ means $\lim \frac{A}{B} = 1$, while   $A \lesssim B$ (resp. $A \gtrsim B$) means $\limsup \frac{A}{B} \leq 1$ (resp. $\liminf \frac{A}{B} \geq 1$).  The Vinogradov notation $A \ll B$ means $A = O(B)$ and $A \asymp B$ means $A\ll B$ and $B \ll A$.  

$d(n)$ denotes the number of divisors of positive integer $n$.  For $z \in \bC$, $d_z(n)$ is the generalized divisor function, defined as the coefficients in the Dirichlet series
\begin{equation}
 \sum_{n} \frac{d_z(n)}{n^s} = \zeta(s)^z, \qquad \Re(s) > 1.
\end{equation}
Thus $d(n) = d_2(n)$. $\mu=d_{-1}$ is the M\"{o}bius function, supported on square-free numbers and given for distinct primes $p_1, ..., p_k$ by $\mu(p_1...p_k) = (-1)^k$. $\omega(n)$ denotes the number of distinct prime divisors of positive integer $n$.

We write $\bS^1 = \{z \in \bC: |z| = 1\}$ and $\bT = \bR/\zed$.  The distance 
$\|\cdot\|_{\bR/\zed}$ on $\bT$ is  inherited from $\bR$.  We write the usual 
character $e:\bT \to \bS^1$, $e(\theta) = e^{2\pi i \theta}$.  Also, $c(\theta) 
= \cos(2\pi \theta)$.  

Contour integrals are abbreviated $\oint_{\Gamma} \cdot dz = \frac{1}{2\pi i} 
\int_{ \Gamma} \cdot dz$.

\section{The resonance method}\label{resonance_function_section}
Consider a family of harmonics $\sF$ with conductor $\sC$.  Thus $\sF = \{n 
\mapsto 
\chi_d(n): d \asymp D\}$ with conductor $D$ in the case of real characters, 
$\sF = \{n \mapsto n^{it}: t \asymp T\}$ with conductor $T$ in the case of 
$\zeta$,  or $\sF = \{n \mapsto \chi(n): \chi \bmod q \text{ non-principal}\}$ 
with conductor $q$ in the case of Dirichlet characters modulo $q$. The 
resonance method uses an auxiliary Dirichlet polynomial to isolate large values 
in the sum of the harmonic.  
Given $\sX \in \sF$, the resonator takes shape $R(\sX) = \sum_{n \leq N} r(n) 
\sX(n)$ 
for some arithmetic function $r$. The resonance method is based upon the 
inequalities
\begin{equation}\label{resonance_ratio}
\min_{\sX \in \sF}f(\sX) \leq  \E_R[f] = \int_{\sX \in \sF} f(\sX) |R(\sX)|^2 
\bigg/ \int_{\sX \in \sF} 
|R(\sX)|^2 \leq \max_{\sX \in \sF}f(\sX).
\end{equation}
The denominator in the formula for the expectation is referred to as the 
normalizing weight.

Let $N$, a power of $\sC$, be a parameter and set $L = \sqrt{\log N \log \log 
N}$.  We take as a starting point the multiplicative function $r(n)$ of 
\cite{S08}, which is supported on square-free numbers and defined at primes by
\begin{equation}
    r(p) = \left\{ \begin{array}{lll} \frac{L}{\sqrt{p} \log p} && L^2 \leq p 
\leq \exp((\log L)^2)\\ 0 && \text{otherwise} \end{array}\right. .
\end{equation}
We also set $
    r'(p) = \frac{r(p)}{1 + r(p)^2}$ and extend $r'$ to $\bN$ multiplicatively, 
with support on square-free numbers. Note that given $N$ is a fixed power of 
$\sC$, this function is essentially optimal for maximizing 
(\ref{resonance_ratio}) in the case that $f(\sX)$ is a function of type $f(\sX) 
= \sum_{n < \sC} \frac{\sX(n)}{\sqrt{n}}$ \cite{S08}.  We record several of the 
further properties of $r$.
   
\begin{lemma}\label{resonance_function_lemma} The function $r(n)$ satisfies the following properties.
 \begin{enumerate}
  \item[i.] \emph{Concentration.} Let $Y < \exp\left(\frac{L}{(\log 
L)^{5}}\right)$ and $Z > \exp\left(L (\log L)^{5}\right)$.  We have, for each 
fixed integer $m > 0$,
  \begin{align}\label{resonance_concentration}
   \sum_{Y \leq n < Z} \frac{r'(n) d_m(n)}{\sqrt{n}} &= \left(1 + 
O_{m}\left(\exp\left(-\frac{L}{(\log 
L)^3}\right)\right)\right)\prod_p\left(1 + \frac{m r'(p)}{\sqrt{p}}\right)
  \\ \notag& = \exp\left((m + o_m(1)) \frac{L}{2\log L}\right).
  \end{align}

  \item[ii.] \emph{Small tails.} With $Z$ as above, for any multiplicative $f$ satisfying $|f(p)| \leq m$, one has 
  \begin{equation}\label{resonance_small_tails_main}
   \left|\sum_{n \geq Z} \frac{r(n) f(n)}{\sqrt{n}}\right| \leq  \exp\left(-(1 + o_m(1)) \frac{\log Z}{(\log L)^3}\right). 
  \end{equation} Also, there is $c > 0$ such that, for multiplicative function 
$g$ 
satisfying $0 \leq g(p) \leq 1$ and for any integer $\ell \geq 1$, and for any  
$Z > N \exp\left(-\frac{\log N}{(\log \log N)^2}\right)$
  \begin{equation}\label{resonance_small_tails_square}
   \sum_{\substack{n < Z\\ (n, \ell) = 1}} r(n)^2 g(n) = \left(1 + O\left(\exp\left( - \frac{cL^2}{(\log L)^4}\right)\right)\right) \prod_{p\nmid \ell} \left(1 + r(p)^2 g(p)\right).
  \end{equation}

  \item[iii.] \emph{Mild roughness of support.}  If $1 < \ell < N$ satisfies $r(\ell)> 0$ then
  \begin{equation}\label{resonance_rough_support}
   \sum_{p|\ell} \frac{1}{p} = O\left(\frac{1}{(\log L)^2}\right).
  \end{equation}

 \end{enumerate}

\end{lemma}

\begin{rem} A number is called `smooth' if it is composed of many small prime factors, and `rough' if it is composed of only larger primes.  
The fact that the resonating function $r(n)$ is concentrated on primes somewhat 
larger than $\log \sC$ is a novel feature of the resonance method and it plays 
an important role especially in the proofs of Theorems 
\ref{large_zeta_theorem} and \ref{dirichlet_L_angle_theorem}.  
\end{rem}

\begin{proof}[Proof of Lemma \ref{resonance_function_lemma}]
Recall our convention that limiting statements are taken with respect to 
growing conductor.
 Let $F(s) = \sum_n \frac{b_n}{n^s}$ be a Dirichlet series with positive coefficients.  `Rankin's trick' refers to the bounds, 
 \begin{equation}
  \forall\; \alpha > 0,\; Y, Z \geq 1,\qquad \sum_{n \leq Y} b_n \leq Y^\alpha F(\alpha), \quad \sum_{n \geq Z} b_n \leq Z^{-\alpha} F(-\alpha). 
 \end{equation}
For $|\alpha| \leq \frac{1}{(\log L)^3}$, say, we have, by partial summation against the prime number theorem,
\begin{align}\notag
\log \sum_n \frac{r'(n) d_m(n)}{n^{\frac{1}{2} + \alpha}} - \log \sum_n 
\frac{r'(n) d_m(n)}{n^{\frac{1}{2}}}&=\log \prod_p \left(1 + \frac{m 
r'(p)}{p^{\frac{1}{2}+\alpha}} \right) - \log \prod_p \left(1 + \frac{m 
r'(p)}{\sqrt{p}} \right) \\&  \sim -m\alpha L \log \log L.
 \\ \notag \log \prod_p\left(1 + \frac{m r'(p)}{\sqrt{p}} \right) &\sim m 
\frac{L}{2\log L}.
 \end{align}
(\ref{resonance_concentration}) and (\ref{resonance_small_tails_main}) follow by 
choosing $\alpha = \pm \frac{1}{(\log L)^3}$ and applying Rankin's trick. 

(\ref{resonance_small_tails_square}) follows by choosing $\alpha = \frac{1}{(\log L)^3}$ in the estimate
\begin{align}
 &\log \prod_p \left(1 + r(p)^2 g(p) p^\alpha\right) - \log \prod_p \left(1 + 
r(p)^2 g(p)  \right) \\ \notag  &\leq  \alpha \left(\log N -(1 + 
o(1))\frac{\log N \log\log\log N}{\log \log N}\right),
\end{align}
which again is valid for $|\alpha| \leq \frac{1}{(\log L)^3}$.

To prove (\ref{resonance_rough_support}), note that $n < N$ in the support of $r$ has at most $\frac{\log N}{2\log L} \lesssim \frac{L^2}{4(\log L)^2}$ prime factors.  Since each has size at least $L^2$, the result follows.

\end{proof}

\subsection{The fractional divisor function}
Our proofs of Theorems \ref{large_zeta_theorem} and \ref{dirichlet_L_angle_theorem} use a short Dirichlet polynomial formed from the fractional divisor function.  This is defined as the coefficients in the Dirichlet series
\begin{equation}
 \sum_n \frac{d_z(n)}{n^s} = \zeta(s)^z, \qquad z \in \bC, \; \Re(s) > 1.
\end{equation}
We use the case $d_{\frac{1}{2}}(n)$.  This is a multiplicative function given at prime powers by
\begin{equation}
 d_{\frac{1}{2}}(p^k) = (-1)^k \binom{-\frac{1}{2}}{k} = \frac{1}{2^k k!}\prod_{i=1}^k (2i-1).
\end{equation}
In particular, 
\begin{equation}\label{decreasing_property}\forall\; k \geq 1,\qquad \frac{1}{2}d_{\frac{1}{2}}(p^{k-1}) \leq d_{\frac{1}{2}}(p^k) \leq d_{\frac{1}{2}}(p^{k-1}).\end{equation}
For $k, n \in \zed_{\geq 1}$ and $ x  \in \bR_{>1}$ we define the restricted divisor function
\begin{equation}
 d_{\frac{1}{2}, 2k, x}(n) = \sum_{\substack{n_1, ..., n_{2k} \leq x\\ n_1\cdots n_{2k} = n}} d_{\frac{1}{2}}(n_1)\cdots d_{\frac{1}{2}}(n_{2k}).
\end{equation}
We have $d_{\frac{1}{2}, 2k, x}(n) \leq d_k(n)$, with equality if $n \leq x$.  
This follows from the Euler product definition of $d_z(n)$ and from positivity 
of $d_{\frac{1}{2}}$.

We refer to (\cite{T95}, p. 184) for asymptotics in sums of products of divisor 
functions,  which are evaluated with use of the Hankel contour.  Note that the 
shape of the asymptotic is determined by the behavior of the relevant Dirichlet 
series at prime values, with the larger prime powers contributing only to the 
leading constant.

The following lemma makes  use of the roughness of support of the resonance function (see iii of Lemma \ref{resonance_function_lemma}).
\begin{lemma}\label{divisor_extraction} Let $r$ be the multiplicative function described above, associated to parameter $N$, and let $X < N$.  Let $\ell_1, \ell_2 < X^{\frac{1}{2}}$, $(\ell_1, \ell_2) = 1$, with $r(\ell_1), r(\ell_2) > 0$. We have the asymptotic evaluations
\begin{equation}\label{one_factor_lower_bound}
 \sum_{n < \frac{X}{\ell_1}} \frac{ d_{\frac{1}{2}}(\ell_1 n) }{n} \asymp d_{\frac{1}{2}}(\ell_1)  (\log X)^{\frac{1}{2}}
\end{equation}
and
\begin{equation}\label{two_factor_lower_bound}
  \sum_{n < \frac{X}{\max(\ell_1, \ell_2)}} \frac{d_{\frac{1}{2}}(\ell_1 n) d_{\frac{1}{2}}(\ell_2 n)}{n} \asymp d_{\frac{1}{2}}(\ell_1)d_{\frac{1}{2}}(\ell_2) (\log X)^{\frac{1}{4}},
\end{equation}
 the lower bound
\begin{equation}\label{divisor_lower_bound}
 \sum_{n < X} \frac{d(\ell_1n)d(\ell_2n)}{n} \gg d(\ell_1)d(\ell_2)(\log X)^4 
\end{equation}
and the upper bound
\begin{equation}\label{two_divisor_upper_bound}
 \sum_{\substack{n, m_1\\ \ell_2 m_1, \ell_1m_1n < X}} \frac{d_{\frac{1}{2}}(\ell_2m_1) d_{\frac{1}{2}}(\ell_1 m_1 n)}{m_1n} \ll  d_{\frac{1}{2}}(\ell_1)d_{\frac{1}{2}}(\ell_2) (\log X)^{\frac{3}{4}}.
\end{equation}

\end{lemma}
\begin{proof}
All of the upper bounds proceed in the same manner, so we just describe the first one.  Due to the support of $r$, $\ell_1$ is square-free, and so, applying both inequalities in (\ref{decreasing_property}) in turn,
\begin{align}
 \sum_{n < \frac{X}{\ell_1}} \frac{ d_{\frac{1}{2}}(\ell_1 n) }{n} &\leq 
d_{\frac{1}{2}}(\ell_1) \sum_{d|\ell_1} 2^{\omega(d)} \sum_{\substack{n < 
\frac{X}{\ell_1}\\ d|n}} \frac{d_{\frac{1}{2}}(n)}{n}\\& \notag \leq 
d_{\frac{1}{2}}(\ell_1) \prod_{p|\ell_1} \left(1 + \frac{2}{p}\right) 
\sum_{\substack{n < \frac{X}{\ell_1}}} \frac{d_{\frac{1}{2}}(n)}{n} \\& 
\notag \ll d_{\frac{1}{2}}(\ell_1) (\log X)^{\frac{1}{2}}  \prod_{p|\ell_1} 
\left(1 + \frac{2}{p}\right).
\end{align}
For the first line, write $d = (\ell,n)$ and $\ell = \ell' d$, and use the 
multiplicativity of $d_{\frac{1}{2}}$ together with (\ref{decreasing_property}).
The product is $1 + o(1)$ by the roughness of the support of $r$.

For the lower bound in (\ref{one_factor_lower_bound}), restrict to $(n, \ell_1)=1$ to obtain the bound
\begin{align}
(\ref{one_factor_lower_bound})& \geq d_{\frac{1}{2}}(\ell_1) \sum_{\substack{n 
\leq \frac{X}{\ell_1}\\ (n, \ell_1) = 1}}  \frac{ d_{\frac{1}{2}}( n) }{n} 
\\ \notag &\geq d_{\frac{1}{2}}(\ell_1)\left( \sum_{n \leq 
\frac{X}{\ell_1}}\frac{ d_{\frac{1}{2}}( n) }{n} - \sum_{p |\ell_1} \sum_{n 
\leq \frac{X}{p\ell_1}}\frac{ d_{\frac{1}{2}}( pn) }{pn} \right)\\ \notag
& \geq  d_{\frac{1}{2}}(\ell_1)\left( \sum_{n \leq \frac{X}{\ell_1}}\frac{ d_{\frac{1}{2}}( n) }{n}\right) \left(1 - \sum_{p|\ell_1} \frac{1}{p}\right).
\end{align}
This suffices, since the sum over $n$ is $\gg (\log X)^{\frac{1}{2}}$, while the sum over $p$ is $o(1)$ by the roughness of support of $r$.  The proof of the lower bound in (\ref{two_factor_lower_bound}) is similar.

\end{proof}

Combining the results of this section, we prove an estimate which will be of later use.
Let $\delta> 0$ be a small constant, and form the convolution
\begin{equation}
 \forall\; n \leq N^{1 + \delta}, \qquad a_n = \sum_{\substack{n_1n_2 = n\\ n_1 \leq N, \; n_2 \leq N^\delta}} r(n_1)\frac{d_{\frac{1}{2}}(n_2)}{\sqrt{n_2}}. 
\end{equation}

\begin{lemma}\label{eval_an_square_lemma}
  Recall that we define $r'(p) = \frac{r(p)}{1 + r(p)^2}$. It holds
 \begin{equation}
  \sum_{n \leq N^{1 + \delta}} a_n^2 \asymp_\delta (\log N)^{\frac{1}{4}} \prod_p \left(1 + r(p)^2\right) \prod_p \left(1 + \frac{r'(p)}{\sqrt{p}}\right).
 \end{equation}
\end{lemma}

\begin{proof}
 Expand $\sum_{n\leq N^{1 + \delta}} a_n^2$, writing $g = \GCD(\ell_1,\ell_2)$ 
and replacing $\ell_1:= \frac{\ell_1}{g}$ and $\ell_2 := \frac{\ell_2}{g}$, to 
obtain
 \begin{align}
  &\sum_g r(g)^2\sum_{\substack{\ell_1, \ell_2 \leq \frac{N}{g}\\ 
(\ell_1,\ell_2) = (\ell_1\ell_2, g) = 1}} r(\ell_1)r(\ell_2) 
\sum_{\substack{n_1, n_2 \leq N^\delta\\ \ell_1n_1 = \ell_2n_2}} 
\frac{d_{\frac{1}{2}}(n_1) d_{\frac{1}{2}}(n_2)}{\sqrt{n_1n_2}}\\ \notag
  &= \sum_g r(g)^2\sum_{\substack{\ell_1, \ell_2 \leq \frac{N}{g}\\ (\ell_1,\ell_2) = (\ell_1\ell_2, g) = 1}} \frac{r(\ell_1)r(\ell_2)}{\sqrt{\ell_1\ell_2}} \sum_{n\leq \frac{N^\delta}{\max(\ell_1,\ell_2)}} \frac{d_{\frac{1}{2}}(\ell_1n) d_{\frac{1}{2}}(\ell_2n)}{n}.
 \end{align}
For $\max(\ell_1, \ell_2) < Z := \exp\left((\log N)^{\frac{2}{3}}\right)$, 
(\ref{two_factor_lower_bound}) of Lemma \ref{divisor_extraction} gives that the 
inner summation is $\asymp_\delta d_{\frac{1}{2}}(\ell_1) 
d_{\frac{1}{2}}(\ell_2) (\log N)^{\frac{1}{4}}.$  The concentration properties 
of the resonator function given in Lemma \ref{resonance_function_lemma} then 
give
\begin{align}
& (\log N)^{\frac{1}{4}} \sum_{\substack{\ell_1, \ell_2 < Z\\ (\ell_1, \ell_2) = 
1}} \frac{r(\ell_1)r(\ell_2)d_{\frac{1}{2}}(\ell_1) d_{\frac{1}{2}}(\ell_2) 
}{\sqrt{\ell_1\ell_2}} \sum_{\substack{g \leq \frac{N}{\max(\ell_1,\ell_2)}\\ 
(g, \ell_1\ell_2)=1}}r(g)^2\\& \notag \sim (\log N)^{\frac{1}{4}} 
\prod_p\left(1 + \frac{r'(p)}{\sqrt{p}}\right) \prod_p (1 + r(p)^2),
\end{align}
so that we obtain the main term. 
To verify this estimate, first replace the sum over $g$ with an Euler 
product at primes co-prime to $\ell_1\ell_2$, bounding the error using 
(\ref{resonance_small_tails_square}) and (\ref{resonance_concentration}). 
Replacing $r$ with $r'$ accounts for the missing factors.  Now extend the sum 
over $\ell_1, \ell_2$ by removing the condition on $Z$. In the tail where 
$\max(\ell_1, \ell_2) > Z$, bound the sum over $n$ trivially by $\ll \log N$, 
and use the small tails property from Lemma \ref{resonance_function_lemma} to 
show that this contribution is negligible.
\end{proof}

\section{Negative truncation of $L(\frac{1}{2}, \chi_d)$}\label{truncation_section}
Our main result of this section, Theorem \ref{negative_truncation_theorem}, 
shows that under GRH a fixed Dirichlet series thought of as a function of the 
character $\chi_d$ of length a small power of $d$ is not sufficient to 
approximate $L\left(\frac{1}{2}, \chi_d\right)$ pointwise.  By way of 
comparison, assuming GRH we check  that in the region $\left(\sigma - 
\frac{1}{2}\right) \log \log d \to \infty$, a polynomial of length any power of 
$d$ suffices.
\begin{proposition}
 Fix a smooth function $\phi : \bR^+ \to [0,1]$ satisfying $\phi \equiv 1$ in a 
neighborhood of 0 and $\phi$ is supported in $[0,1]$.  Let $d > 0$ be a large 
fundamental discriminant and assume the Riemann Hypothesis for $L(s,\chi_d)$.  
There are positive constants $C_1, C_2 > 0$, such that if $(\sigma - 
\frac{1}{2}) > \frac{C_1}{\log \log d}$ and if $x > \exp\left(C_2 (\log 
d)^{2-2\sigma} + \log\log d\right)$ then
 \begin{equation}
  \sum_n \frac{\chi_d(n)}{n^\sigma} \phi\left(\frac{n}{x}\right) \sim L(\sigma, \chi_d) > 0.
 \end{equation}
\end{proposition}

\begin{proof}
Assuming the Riemann hypothesis, the bound, in $\sigma > \frac{1}{2}$
 \begin{equation}
  \left|\Re(\log L(\sigma+it, \chi_d))\right| \leq C \frac{(\log (d+|t|))^{2-2\sigma}}{\log \log (d+|t|)} +  \log\log (d + |t|)
 \end{equation}
is classical.  For the best known constants in this estimate for the case of the Riemann zeta function, see \cite{CC11}.  Write
\begin{equation}
 \tilde{\phi}(w) = \int_0^\infty \phi(x) x^{w-1}dx.
\end{equation}
This function has a simple pole at 0 of residue 1.  Integrating by parts, it satisfies the estimate
\begin{equation}
\forall |\Re(w)|\leq 2, \; \forall A > 0, \qquad |\tilde{\phi}(w)| \ll_A \frac{1}{|w|^A}.
\end{equation}
By Mellin inversion,
\begin{equation}
 \sum_n \frac{\chi_d(n)}{n^\sigma} \phi\left(\frac{n}{x}\right) = \oint_{\Re w = 
1} x^w L(\sigma + w, \chi_d) \tilde{\phi}(w)dw.
\end{equation}
Shift the contour to $\Re(w) = -\frac{C'}{\log \log D} > \frac{1}{2}-\sigma$,  
picking up a residue of $L(\sigma, \chi_d)$ from the pole at 0. For an 
appropriate combination of the $C$ and $C'$, the remaining  integral is 
$o(L(\sigma, \chi_d))$.
\end{proof}

We now prove Theorem \ref{negative_truncation_theorem}.
\subsection{The signed resonance function}
Our resonating function in this section deviates slightly from those used in the 
latter two sections.  Set $x = D^{\frac{2}{9}-\delta}$, $0 < \delta < 
\frac{2}{9}$ for the length of our truncated $L$-function,
\begin{equation}
T(\chi_{8d}) = \sum_n \frac{\chi_{8d}(n)}{\sqrt{n}} \phi\left(\frac{n}{x} 
\right).
\end{equation}
Define resonance parameters $N = \min(D^{\frac{\delta}{2}}, 
x^{\frac{1}{4}})$ and $L = \sqrt{\log N \log \log N}$.  With $Z = xN^2$ 
introduce the resonating polynomial
\begin{equation}
 R(\chi_{8d}) = \sum_{n \leq Z} r^*(n) \chi_{8d}(n).  
\end{equation}
We consider the probability measure on fundamental discriminants of form 
$8d$, $\frac{D}{2} \leq d < D$ given by 
\begin{equation}\label{def_D_expectation}
 \E_R[f] = {\sum_{\substack{\frac{D}{2} \leq d < D}}}\mu(2d)^2 f(\chi_{8d}) 
|R(\chi_{8d})|^2 \bigg/ {\sum_{\substack{\frac{D}{2} \leq d < D}}}\mu(2d)^2  
|R(\chi_{8d})|^2.
\end{equation}

In order to achieve a negative expectation, choose a resonating 
multiplicative function $r^*$ which is the convolution $r^* = r^- \ast r^+$ of  
multiplicative functions supported on small and large primes.  The idea is to 
choose $r^-$ so that its summatory function essentially vanishes, while 
convolution with $r^+$ negatively correlates with a sequence of partial sums of 
$r^-$ which are oscillatory and large.

Let $\sP$ be the set of primes and let the small primes be 
\begin{equation}\sP^- = \sP \cap [L^2, \exp((\log L)^2)].\end{equation}  The 
function $r^-$ is given by the multiplicative function $r$ defined in Section 
\ref{resonance_function_section} multiplied by a phase function.  Write $\bT = 
\bR/\zed$, $c(\theta) = \cos(2\pi \theta)$  and let $\chi: \bT \to [-1,1]$ be 
defined by 
\begin{equation}\label{def_chi}\chi(\theta) = \sgn(c( \theta))( 2|c( 
\theta)|-1) \one\left\{|c(\theta)| > \frac{1}{2}\right\}.\end{equation}
 We set
\begin{equation}
 r^-(p) = \left\{ \begin{array}{lll} r(p) \chi\left(\frac{\log p}{4\log 
L}\right) && p \in \sP^-\\0 && \text{otherwise} \end{array}\right..
\end{equation}
Note that $r^-$ is non-positive at the initial sequence of primes where it is 
non-zero.  Since it is supported on small primes its 
summatory function at large argument behaves like an Euler product, which is 
small.  The phase in $\log p$ is responsible for large partial sums at smaller 
argument, a feature which is detected via $r^+$. 

  The definition of $r^+$ is less explicit because it depends upon fluctuations 
in partial sums of the function $r^-(n)$.  Let $B$, $x^{1-\epsilon} < B < x$ be 
a parameter to be determined, let the large primes be $\sP^+ = \sP \cap 
[\frac{B}{4}, B)$, and set
\begin{equation}
 r^+(p) = \left\{\begin{array}{lll} \frac{\epsilon_p}{\sqrt{p}\log x} && p \in 
\sP^+\\ 0&&\text{otherwise}\end{array}\right.,
\end{equation}
where for $p \in \sP^+$, $\epsilon_p = \pm 1$ is at our disposal.  In 
particular, $r^*$ is given by
\begin{equation}
 r^*(p) = \left\{ \begin{array}{lll} \frac{L \chi(\frac{\log p}{4\log 
L})}{\sqrt{p}\log p} && p \in \sP^-\\
                 \frac{\epsilon_p}{\sqrt{p} (\log x)} && p \in \sP^+\\
                 0 && p \not \in \sP^- \cup \sP^+
                \end{array}\right..
\end{equation}

We note that 
\begin{equation}\label{large_prime_square_sum}
 \sum_{p \in \sP^+} r^+(p)^2 = o(1).
\end{equation}
This feature is used to guarantee that $r^+$ does not contribute significantly 
to the mean square of $R(\chi_{8d})$ appearing in the denominator of 
(\ref{def_D_expectation}). The proof of Theorem 
\ref{negative_truncation_theorem} demonstrates this, and that expectation in 
the numerator is essentially linear and typically large as a function of the 
signs $\epsilon_p$.

\subsection{Determination of expectation}
The characters $\chi_{8d}$ satisfy the following orthogonality relation.
\begin{lemma}[Orthogonality relation, \cite{RS06} Lemma 3.1]
 Let $n$ be square.  We have for all $\epsilon > 0$,
 \begin{equation}\label{diagonal_eval}
  \sum_{\frac{D}{2} < d \leq D} \mu^2(2d) \chi_{8d}(n) = \frac{3D}{\pi^2}  
\prod_{p |2n}\left(\frac{p}{p+1}\right) + O_\epsilon\left(D^{\frac{1}{2} + 
\epsilon}n^\epsilon\right).
 \end{equation}
 If $n$ is not square then
 \begin{equation}
  \sum_{d < D} \mu^2(2d) \chi_{8d}(n)= 
O\left(D^{\frac{1}{2}}n^{\frac{1}{4}}\log n \right).
 \end{equation}
 \end{lemma}

 Using this lemma we give the asymptotic evaluation of the normalizing weight 
and expectation of $T(\chi_{8d})$.
 
 \begin{lemma}
 We have the following asymptotic evaluation of the normalizing weight.
 \begin{equation}
\rNW:= \sum_{\frac{D}{2} < d \leq D} \mu(2d)^2 |R(\chi_{8d})|^2 \sim 
\frac{2D}{\pi^2}  \prod_{p \in \sP^-} \left(1 + \frac{p}{p+1}r^-(p)^2\right). 
 \end{equation}
\end{lemma}

\begin{proof}
Expanding the square,
 \begin{equation}
\sum_{\frac{D}{2} < d \leq D} \mu(2d)^2 |R(\chi_{8d})|^2=  \sum_{n_1, n_2 \leq 
Z} r^*(n_1)r^*(n_2) \sum_{\frac{D}{2} < d \leq D}\mu(2d)^2\chi_{8d}(n_1n_2).
 \end{equation}
Recall that $r^*$ is supported on odd square-free numbers.  A diagonal term 
arises from $n_1 = n_2$, see (\ref{diagonal_eval}).  Define $\tilde{r}(p) = 
\sqrt{\frac{p}{p+1}} r^*(p)$ and extend $\tilde{r}(n)$ multiplicatively, 
supported on odd square-free numbers.  The diagonal term is
\begin{equation}
 \frac{2D}{\pi^2}  \sum_{n \leq Z} \tilde{r}(n)^2 = \frac{2D}{\pi^2}  
\sum_{\substack{n \leq Z\\ (n, \sP^+)=1}} \tilde{r}(n)^2 + O\left(D \sum_{p \in 
\sP^+} \tilde{r}(p)^2 \sum_{n \leq \frac{Z}{p}} \tilde{r}(n)^2\right).
\end{equation}
The leading factor of 2 accommodates the prime 2. 
Note that $Z \leq x^{\frac{3}{2}}$ so that for $p \in \sP^+$, $n \leq 
\frac{Z}{p}$ is free of prime factors from $\sP^+$. Thus,
since $\sum_{p \in \sP^+} \tilde{r}(p)^2 = o(1)$ the error term is negligible 
compared to the main term.  Since $Z >N$ the main term  is $\sim  
\frac{2}{\pi^2} D \prod_{p \in \sP^-} (1 + \tilde{r}(p)^2)$ by 
(\ref{resonance_small_tails_square}) of Lemma \ref{resonance_function_lemma}.

The off-diagonal terms are bounded by
\begin{align}
 &\ll_\epsilon D^{\frac{1}{2}+\epsilon}\sum_{n_1, n_2<Z} 
|r^*(n_1)r^*(n_2)|(n_1n_2)^{\frac{1}{4}}\\\notag &\ll_\epsilon D^{\frac{1}{2} + 
\epsilon}Z^{\frac{3}{2}}\sum_{n_1, n_2} 
\frac{|r^*(n_1)||r^*(n_2)|}{\sqrt{n_1n_2}}.
\end{align}
Since $Z \leq x D^\delta$ and $x = D^{\frac{2}{9}-\delta}$ the off-diagonal is 
$O_\epsilon\left(D^{\frac{5}{6}+\epsilon}\right)$, bounding the sum over $n_1, 
n_2$ using (\ref{resonance_concentration}).
\end{proof}

\begin{proposition}\label{numerator}  Define arithmetic function $a_n$ by the 
Dirichlet series
\begin{align}
\label{def_F_s}F(s) &= \sum_n \frac{a_n}{n^s}\\
\notag&= \prod_{p \in \sP^-}\left(1 + \frac{2 
r^-(p)}{p^{\frac{1}{2}+s}(\frac{p+1}{p} + r^-(p)^2)} 
\frac{p^{2s+1}}{p^{2s+1}-1}+ \frac{1 + r^-(p)^2}{(p^{2s+1}-1)(\frac{p+1}{p} + 
r^-(p)^2)} \right)\\
&\notag\times \frac{1}{2^{2s+1}-1}\prod_{p \not \in \sP^-, \text{ odd}} \left(1 
+ \frac{p}{p+1}\frac{1}{p^{2s+1}-1}\right).
\end{align}
The normalized expectation of $T(\chi_{8d})$ satisfies $\E_R[T(\chi_{8d})] 
\sim \Sigma_1 + O(\Sigma_2) + o(1)$, where
\begin{align}
  \Sigma_1  = \frac{2}{ \log x}\sum_{p \in \sP^+}\frac{\epsilon_p}{p } \sum_{n} 
a_n   \phi\left(\frac{np}{x}\right), \qquad \Sigma_2 =  
\sum_{n}a_n\phi\left(\frac{n}{x}\right).
\end{align}
\end{proposition}

\begin{proof}
Expand the square in the numerator and pass the sum over $d$ inside to obtain
\begin{align}
 &\rNW \cdot \E_R[T(\chi_{8d})] = \sum_{\ell_1, \ell_2 \leq Z} 
r^*(\ell_1)r^*(\ell_2) \sum_{n }\frac{\phi\left(\frac{n}{x}\right)}{\sqrt{n}}  
\sum_{\frac{D}{2} < d \leq D} \mu^2(2d) \chi_{8d}(\ell_1\ell_2n).
\end{align}
Recall that $r^*$ is supported on square-frees.  Pull out the GCD $g= (\ell_1, 
\ell_2)$, writing $\ell_1' = \frac{\ell_1}{g}$, $\ell_2' = \frac{\ell_2}{g}$.   
We obtain a diagonal term coming from $n$ of form $\ell_1' \ell_2' m^2$, see 
(\ref{diagonal_eval}).  This term is
\begin{equation}\label{diagonal}
 \frac{2D}{\pi^2} \sum_{\substack{\ell = \ell_1\ell_2  \\ (\ell_1,\ell_2) = 1}} 
\frac{r^*(\ell)}{\sqrt{\ell}} \sum_{\substack{ g\leq \frac{Z}{\max(\ell_1, 
\ell_2)}\\ (g, \ell)= 1}} r^*(g)^2 \sum_{m }  \frac{\phi\left(\frac{\ell 
m^2}{x}\right)}{m} \prod_{\substack{p | \ell m g\\ \text{odd}}} \frac{p}{p+1}.
\end{equation}
The off-diagonal terms are bounded by 
\begin{align}
 &D^{\frac{1}{2} + \epsilon} \sum_{\ell_1,\ell_2 \leq Z} \sum_{n \leq x} 
\frac{|r^*(\ell_1)| |r^*(\ell_2)|(\ell_1 \ell_2)^{\frac{1}{4}}}{n^{\frac{1}{4}}} 
\\\notag&\ll_\epsilon D^{\frac{1}{2} + \epsilon}x^{\frac{3}{4}} Z^{\frac{3}{2}} 
\sum_{\ell_1, \ell_2} \frac{|r^*(\ell_1) 
r^*(\ell_2)|}{\sqrt{\ell_1\ell_2}}\ll_\epsilon D^{\frac{1}{2} + 
\frac{3\delta}{2} +  \epsilon} x^{\frac{9}{4}}= O_\epsilon\left(D^{1 - 
\frac{3\delta}{4} +\epsilon}\right).
\end{align}

Before proceeding further, we  comment that the diagonal sum (\ref{diagonal}) is 
bounded absolutely by (see Lemma \ref{resonance_function_lemma}, 
(\ref{resonance_concentration}) for the second upper bound)
\begin{align}
 &\ll D \log D \prod_{p \in \sP^+ \cup \sP^-} \left(\left(1 + 
\frac{2|r^*(p)|}{\sqrt{p}}\right)\left(1 + 
\frac{p}{p+1}r^*(p)^2\right)\right)\\& \notag \leq \rNW 
\exp\left(O\left(\sqrt{\frac{\log D}{\log \log D}}\right)\right) ,
\end{align}
so that, even though the sum contains terms of differing sign, we may make 
relative errors on the order of 
\begin{equation}1 + O\left(\exp\left(-\sqrt{\log D}\right)\right)\end{equation} 
within individual terms without altering the final asymptotic.

Bearing this in mind, we split the diagonal term (\ref{diagonal}) into two sums 
$\Sigma_1^0 + \Sigma_2^0$ according as $\ell$ does or does not have a factor $p 
\in \sP^+$ (the support of $\phi$ guarantees that it has at most one).  In the 
former case, we have 
\begin{align}
 \Sigma_1^0 &= \frac{2D}{\pi^2}\sum_{p \in \sP^+} \frac{2r^+(p)}{\sqrt{p}} 
\sum_{\ell < \frac{2x}{p}}   \frac{r^-(\ell)}{\sqrt{\ell}} \sum_{m } 
\frac{\phi\left(\frac{\ell p m^2}{x}\right)}{m}  \sum_{\ell_1 \ell_2 = 
\ell}\sum_{\substack{ g \leq \frac{Z}{p \ell_1} \\ (g, \ell) = 1}} r^*(g)^2 
\prod_{\substack{p' | \ell m g\\ \text{odd}}} \frac{p'}{p'+1}
\end{align}
Notice that $g < \frac{Z}{p\ell_1}$ implies $g \leq x^{\frac{1}{2} + \epsilon}$ 
so that $g$ has no factors from $\sP^+$.  Also   
\begin{equation}\frac{Z}{p\ell_1} \geq \frac{Z}{2x}=\frac{N^2}{2},\end{equation} 
so that (\ref{resonance_small_tails_square}) of Lemma 
\ref{resonance_function_lemma} implies that to within admissible relative error,
\begin{align}
 &\sum_{\substack{ g \leq \frac{Z}{p \ell_1} \\ (g, \ell) = 1}} \left(r^-(g)^2 
\prod_{\substack{p' | g\\ p' \nmid m}} \frac{p'}{p'+1}\right)\sim
 \prod_{\substack{p' \in \sP^-\\ p' \nmid \ell m}} \left(1 + 
\frac{p'}{p'+1}r^*(p')^2\right) \prod_{\substack{p' \in \sP^-\\ p' |m, p' \nmid 
\ell}} (1 + r^*(p')^2).
\end{align}
  Substituting this evaluation into $\Sigma_1^0$ and dividing by $\rNW$ we obtain
\begin{align}
 \frac{\Sigma_1^0}{\rNW} + o(1) = \frac{2}{\log x} \sum_{p \in \sP^+} \frac{\epsilon_p}{p} \sum_\ell \frac{d(\ell) r^-(\ell)}{\sqrt{\ell}\prod_{p'|\ell} \left(\frac{p'+1}{p'}+ r^-(p')^2\right)}\sum_m \frac{b(\ell,m)}{m} \phi\left(\frac{p\ell m^2}{x} \right),
 \end{align}
 where
 \begin{align}
  b(\ell,m) = \prod_{\substack{p' \in \sP^-\\ p'|m, p' \nmid \ell}}\left(\frac{1 + r^-(p')^2}{\frac{p'+1}{p'} + r^-(p')^2} \right) \prod_{\substack{p' |m\\ p' \not \in \sP^-, \text{ odd}}}\left(\frac{p'}{p'+1} \right).
\end{align}
Comparing this with the Dirichlet series $F(s)$, we see that we obtain 
  $\Sigma_1$ from the Proposition.

It remains to treat the sum $\Sigma_2^0$, and this we do by splitting the sum further as $\Sigma_2^0 = \Sigma_2^1 + \Sigma_2^2,$ depending on whether or not $g$ has a factor from  $\sP^+$.  We first handle the case that $g$ does not contain such a factor, which we call $\Sigma_2^1$.  We have 
\begin{equation}
 \Sigma_2^1 = \frac{2D}{\pi^2} \sum_{\substack{\ell < 2x\\ (\ell, \sP^+) = 1}} \frac{r^-(\ell)d(\ell)}{\sqrt{\ell}}  \sum_{\substack{ (g, \ell) = 1\\(g, \sP^+) = 1}} r^-(g)^2 \sum_{m } \frac{\phi\left(\frac{\ell m^2}{x}\right)}{m} \prod_{\substack{p |\ell m g\\ \text{odd}}} \frac{p}{p+1}.
\end{equation}
Instead of the divisor function $d(\ell)$ we should have included a sum over $\ell_1 \ell_2 = \ell$, with the restriction on $g$ that $g < \frac{Z}{\max(\ell_1, \ell_2)}$, but this may be removed, again by (\ref{resonance_small_tails_square}) of Lemma \ref{resonance_function_lemma}.  Writing the sum over $g$ as a product and dividing by $\rNW$ we arrive at $\Sigma_2$.

It remains to treat $\Sigma_2^2$.  Here we have
\begin{align}
 \Sigma_2^2 &= \frac{2D}{\pi^2} \sum_{p \in \sP^+} r^+(p)^2 \sum_{\substack{\ell_1 \ell_2 = \ell < 2x\\ (\ell, \sP^+) = 1}} \frac{r^-(\ell)}{\sqrt{\ell}} \sum_{m } \frac{\phi\left(\frac{\ell m^2}{x}\right)}{m}  \sum_{\substack{g \leq \frac{Z}{p \max(\ell_1, \ell_2)}\\(g, \ell)=1}} r^-(g)^2 \prod_{\substack{p' | \ell m g\\ \text{odd}}} \frac{p'}{p'+1}.
\end{align}
We wish to again replace the sum over $g$ with a product, but this is only valid 
if the  sum is sufficiently long.  Since $\ell$ restricts the length of $g$, we 
first bound in absolute value the contribution of all terms with 
\begin{equation}\ell > \min\left(D^{\frac{\delta}{4}}, x^{\frac{1}{8}}\right)= 
\sqrt{N}.\end{equation}  This is negligible:
\begin{align}
 &\ll D (\log D)^{-2} \prod_{p' } (1 + r^*(p')^2) \sum_{ \substack{\ell > \sqrt{N}\\ (\ell, \sP^+) = 1}} \frac{|r^-(\ell)|d(\ell)}{\sqrt{\ell}}\\&\ll \rNW  \exp\left(-c \frac{\log D}{(\log \log D)^3}\right)
\end{align}
by applying (\ref{resonance_small_tails_main}) of Lemma \ref{resonance_function_lemma}.    

Restrict to terms with $\ell < \sqrt{N}$. The condition $g \leq 
\frac{Z}{p\max(\ell_1, \ell_2)}$ implies $g \leq x^{\frac{1}{2} + \epsilon},$ so 
$g$ is free of prime factors from $\sP^+$.  Also, 
\begin{equation}\frac{Z}{p\max(\ell_1, \ell_2)}\geq \frac{Z}{p\sqrt{N}} > 
N^{\frac{3}{2}},\end{equation} so that we may now replace the sum over $g$ with 
a product by again applying (\ref{resonance_small_tails_square}) of Lemma 
\ref{resonance_function_lemma}.  

Having done this, we may now reinsert those terms $\sqrt{N}<\ell<2x$ that are composed solely of prime factors from $\sP^-$, again with negligible error.  With these adjustments, we now find that
\begin{equation}
 \Sigma_2^2/\rNW =  \Sigma_2 \times \sum_{p \in \sP^+} r^+(p)^2,
\end{equation}
and since $\sum_{p \in \sP^+} r^+(p)^2 = o(1)$, the proof is complete. 
\end{proof}
 
 Before proceeding further, we prove the following properties of the Dirichlet series $F(s)$, whose definition we recall for convenience.
 \begin{align}
F(s)&= \prod_{p \in \sP^-}\left(1 + \frac{2 
r^-(p)}{p^{\frac{1}{2}+s}(\frac{p+1}{p} + r^-(p)^2)} 
\frac{p^{2s+1}}{p^{2s+1}-1}+ \frac{1 + r^-(p)^2}{(p^{2s+1}-1)(\frac{p+1}{p} + 
r^-(p)^2)} \right)\\ \notag
&\times \frac{1}{2^{2s+1}-1}\prod_{p \not \in \sP^-, \text{ odd}} \left(1 + \frac{p}{p+1}\frac{1}{p^{2s+1}-1}\right),
\end{align}
 \begin{lemma}\label{r_star_properties}
 The Dirichlet series $F(s)$ factors as $\zeta(2s+1)G(s)H(s)$ with 
 \begin{equation}
  H(s) = \prod_{p \in \sP^-} \left(1 + \frac{2r^-(p)}{p^{\frac{1}{2} + s}\left(\frac{p+1}{p} + r^-(p)^2 \right)}\right).
 \end{equation}
 These functions satisfy the following properties.
 \begin{enumerate}
  \item [i.] \emph{$G(s)$ is approximately constant.} Define $\log G(s)$ by 
continuous variation from $+\infty$. For $\Re(s) \geq \sigma_0 > 
-\frac{1}{2}$, $|\log G(s)|$ is bounded by an absolute constant depending only 
upon $\sigma_0$. 
  \item [ii.] \emph{$H(s)$ approximately vanishes at 0.} There is constant 
$c_1>0$ such that, for $|s| \leq \frac{1}{(\log L)^{1+\epsilon}}$, 
 \begin{equation}
  \Re \log H(s) \leq -c_1 \frac{L}{\log L}.
 \end{equation}
 \item [iii.] \emph{$H(s)$ is not too large.} There is a constant $c_2 > 0$ such that for $\Re(s) \geq \frac{-1}{(\log L)^{1+\epsilon}}$
 \begin{equation}
  \Re \log H(s) \leq c_2 \frac{L}{\log L}.
 \end{equation}
 \item [iv.] \emph{Approximate saddle point.} There is $c_3 > 0.36845$ such that for $\sigma = \frac{1}{(\log x)^2}$ and all $t$ such that $\left|t - \frac{\pi}{ 2\log L}\right| \leq \frac{1}{(\log L)^{1+\epsilon}}$, 
 \begin{equation}
  \log \frac{|F(\sigma + it)|^2}{\prod_{p \in \sP^-} \left(1 + \frac{2|r^-(p)|}{\sqrt{p}}\right)} \geq (c_3 + o(1)) \frac{L}{2\log L}.
 \end{equation}
 \item[v.] \emph{$L^1$ control.} For $y \geq 1$ we have
 \begin{equation}
  \sum_{n \leq y} |a_n| \ll \log (1 + y) \exp\left(\frac{L}{2\log L}\right).
 \end{equation}

 \end{enumerate}

\end{lemma}
\begin{proof}
i. Let the term at $p \in \sP^-$ of $H(s)$ be $H_p(s)$.  We have
\begin{align}
  G(s) =& \prod_{p \in \sP^-}\left(1  - \frac{1}{p^{2s+2}}\frac{1 
}{\frac{p+1}{p} + r^-(p)^2}\frac{1}{H_p(s)}\right) \prod_{p \not \in \sP^-, 
\text{ odd}}\left(1 -\frac{1}{(p+1)p^{2s+1}} \right).
\end{align}
For $\Re(s) \geq \sigma_0 > \frac{-1}{2}$, $H_p(s) =1+ o(1)$.  In particular, the product defining $G(s)$ is absolutely convergent, and its logarithm is bounded by an absolute constant.

ii. When $|s| < \frac{1}{(\log L)^{1 + \epsilon}}$, partial summation against the prime number theorem gives
\begin{align}
 \log H(s) &= L\int_{L^2}^{\exp((\log L)^2)}\frac{\chi\left(\frac{\log x}{4\log L}\right)}{x (\log x)^2} dx + o\left(\frac{L}{\log L}\right).
\\\notag &= \frac{L}{2\log L} \int_1^\infty 
\frac{\chi\left(\frac{u}{2}\right)du}{u^2} + o\left(\frac{L}{\log L}\right).
 \end{align}
Recall that $\chi$ is given by \begin{equation}\chi(\theta) = \sgn(c( \theta))( 
2|c( \theta)|-1) \one\left\{|c(\theta)| > \frac{1}{2}\right\}.\end{equation} 
 In each interval $[2n-1, 2n+1]$, $n \geq 1$ the integral against $u$ is negative, as follows by convexity.
 
 iii. This follows as in ii. by partial integration against the prime number theorem.
 
 iv. We have $\log F(s) = O(\log \log x) + \log H(s)$.  By the same partial summation argument as above, 
 \begin{align}
  &2 \Re \log H(s) - \log \prod_{p \in \sP^-} \left(1 + 
\frac{2|r^-(p)|}{\sqrt{p}} \right) \\\notag&= \frac{L}{2\log L} \int_1^\infty 
\frac{2 c\left(\frac{u}{2}\right) \chi\left(\frac{u}{2}\right) - 
\left|\chi\left(\frac{u}{2}\right)\right|}{u^2} du + o\left(\frac{L}{\log 
L}\right)\\\notag
  &= \frac{L}{2\log L} \int_1^\infty \frac{(2 \left|c\left(\frac{u}{2}\right)\right| -1)^2 \one\{\|u\|_{\bR/\zed} \leq \frac{1}{3}\}}{u^2} du + o\left(\frac{L}{\log L}\right).
 \end{align}
 The numerical value of the integral is greater than $0.36845$.

 v. Write $H(s) = \sum_n \frac{b_n}{n^s}$ and $\zeta(2s+1)G(s) = \sum_n \frac{c_n}{n^s}$.  Then 
 \begin{equation}
  \sum_{n \leq y}|a_n| \leq \sum_n |b_n| \sum_{n \leq y} |c_n| \ll \exp\left(\frac{L}{2\log L}\right) \sum_{n \leq y} |c_n|.
 \end{equation}
The last sum is $\ll 1+\log y$, as can be checked by Perron summation.
\end{proof}
 
\subsection{Negative values from the convolution}
The remainder of the proof of Theorem \ref{negative_truncation_theorem} is concerned with the smooth sums 
\begin{align}\label{def_S}
  S(y)& := \sum_{n} a_n\phi\left(\frac{n}{y}\right).
 \end{align}
 Note  that the two quantities from Proposition \ref{numerator} are expressed as
 \begin{equation}
  \Sigma_1 = \frac{2}{\log x}\sum_{p \in \sP^+} \frac{\epsilon_p}{p}S\left(\frac{x}{p}\right), \qquad \Sigma_2(x) = S(x).
 \end{equation}
 
 We also write $S^*(y) = \sum_{n \leq y} |a_n|$
for the sharp truncation.

 \begin{proposition}\label{sigma_2_prop}
Uniformly in $y \geq 2$ there exist $C_1, C_2 > 0$ such that 
 \begin{align} S(y) &\ll \log y\exp\left(-C_1\sqrt{\frac{\log D}{\log \log 
D}}\right) + \exp\left( - \frac{\log y}{(\log \log D)^2} + C_2 \sqrt{\frac{\log 
D}{\log \log D}}\right).
 \end{align}
 In particular, $\Sigma_2 = o(1)$.
 Furthermore, uniformly in $y \geq 0$ there is $C>0$ such that
 \begin{equation}\label{derive_S_y}
  \left|y \frac{\partial}{\partial y} S(y)\right| \ll_\phi S^*(y), \qquad S^*(y) \leq \log(1 + y) \exp\left(C\sqrt{\frac{\log D}{\log \log D}}\right).
 \end{equation}

\end{proposition}

\begin{proof}
The bound for the derivative of $S(y)$ follows on differentiating under the summation and considering the support of $\phi$.  The bound for $S^*(y)$ is given in Lemma \ref{r_star_properties}.

We now estimate $S(y)$.  The conditions on $\phi$ guarantee that its Mellin 
transform $\tilde{\phi}$ has a single simple pole at 0 with residue 1, and that 
for $s$ such that $ |\Re(s)|\leq 1$ and $|\Im(s)|\geq 1$ it satisfies the decay 
property \begin{equation}\forall A>0, \qquad \left|\tilde{\phi}(s)\right| \ll_A
|s|^{-A}.\end{equation}   Therefore, by Mellin inversion,
\begin{equation}\label{sigma_2_contour}
 S(y) =\oint_{\Re(s) = \frac{1}{\log x} } y^s \tilde{\phi}(s) F(s) ds.
\end{equation}

Shift the contour to the line $\Re(s) = \frac{-1}{(\log \log D)^2}$, writing
\begin{align}
 S(y) & =  \oint_{|s| = \frac{1}{\log \max(x,y)}} \zeta(2s+1)G(s)H(s) y^s 
\tilde{\phi}(s) ds \\\notag & \qquad\qquad +  \oint_{\Re(s) = \frac{-1}{(\log 
\log D)^2}} \zeta(2s+1)G(s)H(s) y^s \tilde{\phi}(s) ds
\end{align}
where the first term captures the residue at 0. Recall that $|\log G(s)|$ is 
uniformly bounded in $\Re(s) > \frac{-1}{4}$.  Using the rapid decay of  
$\tilde{\phi}$ to bound the vertical contour and estimating 
$\left|\tilde{\phi}(s)\right|, |\zeta(1 + 2s)| \ll \frac{1}{|s|}$ on the 
circular contour, we find that
\begin{align}\label{S_2_terms}
 |S(y)| &\ll (\log \max(x,y)) \sup_{|s| = \frac{1}{\log \max(x,y)}} |H(s)|\\& \qquad \qquad + \frac{(\log \log D)^2}{\exp\left(\frac{\log y}{(\log \log D)^2}\right)} \sup_{\Re(s) =\frac{-1}{(\log \log D)^2}} |H(s)|. \notag
\end{align}

By the estimates proven in Lemma \ref{r_star_properties} there are constants 
$c_1, c_2 > 0$ such that \begin{align}&\sup_{|s| = \frac{1}{\log \max(x,y)}} 
|H(s)| < \exp\left(-c_1 \sqrt{\frac{\log D}{\log \log D}}\right),  
\\\notag &\sup_{\Re(s) =\frac{-1}{(\log \log D)^2}} |H(s)| < \exp\left( c_2 
\sqrt{\frac{\log D}{\log \log D}}\right),\end{align} completing the estimate.
\end{proof}

For $p \in \sP^+$ we choose $\epsilon_p = -\sgn(S(\frac{x}{p}))$, so that 
\begin{equation}
 \Sigma_1 = - \frac{2}{\log x}\sum_{\frac{B}{4} \leq p < B} \frac{1}{p}\left|S\left(\frac{x}{p}\right)\right|
\end{equation}
By partial summation against the prime number theorem, we obtain
\begin{equation}
 \Sigma_1 = -\frac{2}{\log x} \int_{\frac{B}{4}}^B \left|S\left(\frac{x}{t}\right)\right|\frac{dt}{t \log t} +o(1).
\end{equation}

We make one further reduction. For technical reasons it is convenient to replace 
$\phi$ with a function of compact support on $\bR^+$, so set $\psi(x) = \phi(x) 
- \phi\left(\frac{x}{2}\right)$, and define $\tilde{S}(y) = \sum_n a_n 
\psi\left(\frac{n}{y}\right).$  Noting $\left|\tilde{S}(y)\right| \leq |S(y)| + 
\left|S\left(2y\right)\right|$ gives
\begin{equation}
 \Sigma_1 \leq -\frac{1}{\log x} \int_{\frac{B}{2}}^B \left|\tilde{S}\left(\frac{x}{t}\right)\right|\frac{dt}{t \log t} +o(1).
\end{equation}
Replacing $\frac{x}{t}=: y$ we complete the proof of Theorem 
\ref{negative_truncation_theorem} by choosing $B=A$ in the following 
proposition.
\begin{proposition}\label{large_S}
 Let $\eta = c_3 \min\left(\frac{\delta}{2}, \frac{1}{18} - 
\frac{\delta}{4}\right)$ where $c_3 > 0.36845$ in the constant of Lemma 
\ref{r_star_properties}.  There 
exists $2 \leq A \leq U := \exp(\sqrt{\log x}(\log \log x)^2) $, such that
 \begin{equation}
  \int_{\frac{A}{2}}^A \left|\tilde{S}(y)\right| \frac{dy}{y} \geq 
\exp\left(\sqrt{\frac{\eta \log D}{\log \log D}}\right)
 \end{equation}

\end{proposition}

For the proof, we appeal to the following modified version of Gallagher's large sieve (see e.g. \cite{B74}, p. 29).  
\begin{lemma}\label{gallagher_lemma}
 Let $\psi(x) = \phi(x) - \phi(\frac{x}{2})$ as above.  Define also $\psi_\sigma(x) = x^\sigma \psi(x)$.  Let $P(t) = \sum_{n } a_n n^{-it}$ be a Dirichlet series for which $\sum |a_n| < \infty$.  There exists a constant $\alpha= \alpha(\psi) > 0$  such that, uniformly in $|\sigma| < \frac{1}{2}$, 
 \begin{equation}
  \int_{-\alpha}^\alpha |P(t)|^2 dt \ll_\psi \int_{\frac{1}{2}}^{\infty} \left|\sum_n a_n \psi_\sigma\left(\frac{n}{y}\right)\right|^2 \frac{dy}{y}.
 \end{equation}

\end{lemma}
\begin{proof}
Note that $\psi$ has compact support in $\bR^+$ since $\phi$ has compact support on $\bR$ and $\phi \equiv 1$ on a neighborhood of zero.  

The right hand side is equal to
 \begin{equation}\label{RHS_sieve}
  \sum_{n_1, n_2} a_{n_1}\overline{a_{n_2}} \int_0^\infty \psi_\sigma\left(\frac{n_1}{y}\right)\psi_\sigma\left(\frac{n_2}{y}\right) \frac{dy}{y}.
 \end{equation}
Set $f_\sigma(u) = \psi_\sigma(e^{-u})$ and set $H_\sigma(x) = 
\int_{-\infty}^\infty f_\sigma(u) f_\sigma(u+x) du$.  The Fourier transform of 
$H_\sigma$ is given by $\hat{H_\sigma}(\xi) = 
\left|\hat{f_\sigma}(\xi)\right|^2$, from which it follows that 
(\ref{RHS_sieve}) is given by
\begin{align}
\sum_{n_1, n_2} a_{n_1}\overline{a_{n_2}} H_\sigma\left(\log \frac{n_1}{n_2}\right) &= \sum_{n_1, n_2} a_{n_1}\overline{a_{n_2}} \int_{-\infty}^\infty |\hat{f_\sigma}(\xi)|^2 \left(\frac{n_1}{n_2}\right)^{2\pi i \xi} d\xi
\\\notag &\gg \inf_{|\xi| < 2\pi \alpha, |\sigma| \leq \frac{1}{2}} 
\left|\hat{f_\sigma}(\xi)\right|^2 \times \int_{-\alpha}^\alpha |P(t)|^2 dt.
\end{align}
Let $f$ have support in $[-C, C]$.  Then choose $\alpha \leq \frac{1}{10 \pi 
C}$, say, to guarantee that $\left|\hat{f_\sigma}(\xi)\right|$ is bounded below 
by a  constant depending at most on $\psi$.
\end{proof}

\begin{proof}[Proof of Proposition \ref{large_S}]
Write, as usual,  $F(s) = \sum_n \frac{a_n}{n^s}$.  Let $\sigma = \frac{1}{(\log 
x)^2}$.  We apply Lemma \ref{gallagher_lemma} to the Dirichlet series 
\begin{equation}P(t) = F(\sigma + it) \end{equation} with function 
$\psi_{\sigma}$.  This yields
\begin{equation}
 \int_{-\alpha}^\alpha |F(\sigma + it)|^2 dt\ll \int_{\frac{1}{2}}^\infty 
\left|\tilde{S}(y)\right|^2 \frac{dy}{y^{1 + 2\sigma}}.
\end{equation}
Bounding $\left|\tilde{S}(y)\right| \leq |S(y)| + |S(2y)|$, and applying the 
bound 
\begin{equation}
S(y) \ll \log y\exp\left(-C_1\sqrt{\frac{\log D}{\log \log D}}\right) + \exp\left( - \frac{\log y}{(\log \log D)^2} + C_2 \sqrt{\frac{\log D}{\log \log D}}\right)
\end{equation}
proven in Proposition \ref{sigma_2_prop}, we have  (recall $U = \exp(\sqrt{\log x}(\log \log x)^2)$)
\begin{equation}
 \int_{\frac{1}{2}}^\infty \left|\tilde{S}(y)\right|^2 \frac{dy}{y^{1 + 
2\sigma}} \ll \int_{\frac{1}{2}}^U \left|\tilde{S}(y)\right|^2 \frac{dy}{y} + 
o(1) 
\end{equation}
By dyadic decomposition, we obtain that there is $A$, $2 \leq A \leq U$ such that
\begin{equation}
 \int_{\frac{A}{2}}^A \left|\tilde{S}(y)\right| \frac{dy}{y}\gg \frac{1}{\log U 
\sup_{y <U}\left|\tilde{S}(y)\right|}\int_{-\alpha}^\alpha |F(\sigma + it)|^2 
dt - O(1).
\end{equation}
Since (see the proof of v. of Lemma \ref{r_star_properties}) 
\begin{equation}\left|\tilde{S}(y)\right| \ll \log y \sum_{(n, \sP^+)=1} 
\frac{d(n)\left|\tilde{r}(n)\right|}{\sqrt{n}},\end{equation} the proof follows 
from the 
estimate, for all $t$ such that $\left|t - \frac{\pi}{ 2\log L}\right| \leq 
\frac{1}{(\log L)^{1+\epsilon}}$, 
 \begin{equation}
  \log \frac{|F(\sigma + it)|^2}{\prod_{p \in \sP^-} \left(1 + \frac{2|r^-(p)|}{\sqrt{p}}\right)} \geq (c_3 + o(1)) \frac{L}{2\log L}.
 \end{equation}  given in
item iv. of Lemma \ref{r_star_properties}.
\end{proof}
\section{The argument of large values of $\zeta\left(\frac{1}{2}+it\right)$}
The method of Kalpokas, Korolev and Steuding \cite{KKS13} for treating $\zeta\left(\frac{1}{2} + it\right)$ at points where it has a prescribed angle makes essential use of the fact that the argument of $\zeta\left(\frac{1}{2} + it\right)^2$ is a simple function to describe.  For $\theta \in [0, \pi)$, denote 
\begin{align}
T_\theta &= \left\{t \in \bR: \arg\left(\zeta\left(\frac{1}{2} + 
it\right)\right) \equiv \theta \bmod \pi\right\}\\ \notag
 T_{\theta, +} &= \left\{t \in \bR: \arg\left(\zeta\left(\frac{1}{2} + 
it\right)\right) = \theta\right\}\\ \notag
 T_{\theta, -} &= \left\{t \in \bR: \arg\left(\zeta\left(\frac{1}{2} + it\right)\right) = \theta + \pi\right\}.
 \end{align}
Let $\bU_{T,\theta}$ be the uniform probability measure on $T_\theta \cap 
[0,T]$.
Up to boundedly many exceptions all contained within a compact neighborhood of $\frac{1}{2}$, $T_\theta$ is exactly the solution set of
\begin{equation}
\Delta(s) = \frac{\zeta\left(\frac{1}{2}+s\right)}{\zeta\left(\frac{1}{2}-s\right)} = \pi^{s}\frac{\Gamma\left(\frac{1}{4}-\frac{s}{2}\right)}{\Gamma\left(\frac{1}{4} + \frac{s}{2}\right)} = e^{2i\theta}.
\end{equation}
Using this to express the discrete moments as a contour integral, \cite{KKS13} shows 
\begin{align}
 \E_{\bU_{T, \theta}} \left[\left|\zeta\left(\frac{1}{2} + it\right)\right|^3 
\right] &\gg (\log T)^{\frac{9}{4}}, \\ \notag \E_{\bU_{T,\theta}} \left[ 
\zeta\left(\frac{1}{2} + it\right)^3 \right] &= O((\log T)^2)
\end{align}
from which it follows
\begin{align}
 &\E_{\bU_{T,\theta}}\left[ \left|\zeta\left(\frac{1}{2} + it\right)\right|^3 
\cdot\one\left\{\arg\left(\zeta\left(\frac{1}{2} + 
it\right)\right)=\theta\right\}\right]\\\notag& = 
\frac{1}{2}\E_{\bU_{T,\theta}}\left[\left|\zeta\left(\frac{1}{2} + 
it\right)\right|^3 + e^{-3i\theta} \zeta\left(\frac{1}{2} + it\right)^3\right] 
\gg (\log T)^{\frac{9}{4}},
\end{align}
and similarly for the expectation restricted to points at which $\arg\left(\zeta\left(\frac{1}{2} + it\right)\right) = - \theta.$

Following \cite{S08} we augment this argument by weighting the expectations with a resonating Dirichlet polynomial.  We also estimate the first moment rather than the third, which makes a technical simplification to the argument.  

 Set $H = T/(\log T)^2$.
 Introduce probability measure $w_{T,\theta}(t)$ on $T_\theta$ given by
\begin{equation}
 w_{T,\theta}(t) = \frac{\left|R\left( it\right)\right|^2}{\cosh(\frac{t-T}{H})} \bigg/ \sum_{t' \in T_\theta} \frac{\left|R( it')\right|^2}{\cosh(\frac{t'-T}{H})}
\end{equation}
with $R(s)$ a resonating polynomial.  Let $r$ be the multiplicative function of Section \ref{resonance_function_section} with parameter $N = T^{1-2\xi}$ and define
\begin{equation}
 R^*(s) = \sum_{n < N}\frac{r(n)}{n^s}, \qquad A_{\frac{1}{2}}(s) = \sum_{n < T^{\xi}} \frac{d_{\frac{1}{2}}(n)}{n^s}.
\end{equation}
Our resonating polynomial is
\begin{equation}\label{def_R}
 R(s) = R^*(s) A_{\frac{1}{2}}\left(\frac{1}{2} + s\right) =: \sum_{n < T^{1-\xi}} \frac{a_n}{n^s}.
\end{equation}
We  show the following pair of estimates.
\begin{proposition}\label{main_prop}
 We have
 \begin{equation}\label{absolute_lower_bound}
 \E_{w_{T,\theta}}\left[ \left|\zeta\left(\frac{1}{2} + it\right)\right|\right] \gg_\xi (\log T)^{\frac{3}{4}} \prod_p \left(1 + \frac{ r'(p)}{\sqrt{p}}\right)
\end{equation}
and
\begin{equation}\label{signed_upper_bound}
\left|\E_{w_{T,\theta}} \left[\zeta\left(\frac{1}{2} +it\right) \right] \right| \ll_\xi (\log T)^{\frac{1}{2}} \prod_p \left(1 + \frac{ r'(p)}{\sqrt{p}}\right).
\end{equation}
\end{proposition}

Since \begin{equation}\prod_p \left(1 +\frac{ r'(p)}{\sqrt{p}}\right) = 
\exp\left(\sqrt{\frac{\log N}{\log \log N}}\right)\end{equation}
letting first $T \to \infty$ then $\xi \downarrow 0$ we obtain  Theorem 
\ref{large_zeta_theorem}.

\subsection{Sums involving Gram points} Throughout the remainder of this 
argument we let $\epsilon>0$ be an arbitrarily small fixed constant. We write 
$\Gamma_\epsilon$ for the contour $\Re(s) = \frac{1}{2} + \epsilon$  plus 
$\Re(s) = -\frac{1}{2} - \epsilon$, oriented positively with respect 
to $s = 0$. To within a negligible error, the sums that we need 
over the set $T_\theta$ may be expressed as contour integrals over 
$\Gamma_\epsilon$ 
against the kernel $\frac{\Delta'(s)}{(\Delta(s) - e^{2i\theta})}\frac{ds}{ \cos\left(\frac{ iT - s}{H}\right)}$. 
The following lemma allows us to evaluate integrals of this type.  
\begin{lemma}\label{delta_integral_lemma}
 Let $T$ be large, let $1 \leq m, n$ and assume $m < T^{1-\xi}$. Let $\epsilon > 0$ arbitrary.  We have the following evaluation of integrals.
 
 For any $\omega \in \bS^1 = \{z \in \bC: |z|=1\}$  and for any $A > 0$,
 \begin{align}\label{k_bigger_0}
  \oint_{\Re(s) = \frac{1}{2} + \epsilon} \left(\frac{m}{n}\right)^s \frac{\Delta'}{\Delta}(s) \frac{\Delta(s)}{1-\omega \Delta(s)} \frac{ds}{\cos\left(\frac{ iT - s}{H}\right)} &= O_{\xi, A}\left(T^{-A}\right).
  \end{align}
 Also,
\begin{align}
  \label{k_equal_0}
  &\oint_{\Re(s) = \frac{1}{2} + \epsilon} \left(\frac{m}{n}\right)^s \frac{\Delta'}{\Delta}(s) \frac{ds}{\cos\left(\frac{iT \pm s}{H}\right)} \\\notag&=\delta_{m=n} \left\{-\int_{t \geq 20} \frac{\log \frac{t}{2\pi} + O(\frac{1}{t}) }{\cosh\left(\frac{t-T}{H}\right)}dt \right\} + O_{\xi, A}\left(T^{-A}\right).
  \end{align}
\end{lemma}

\begin{proof}
We use the following consequences of Stirling's formula, which are valid for $|t| >1$
\begin{align}
 \forall\;  \sigma > 0, \qquad &\Delta\left(\sigma + it\right) = 
O_\sigma\left(|t|^{-\sigma}\right), \\ \notag
  &\frac{\Delta'}{\Delta}\left(it\right) = - \log \frac{|t|}{2\pi} + 
O\left(\frac{1}{|t|}\right)\\ \notag
  \forall\; j \geq 1, \qquad & \frac{d^j}{dt^j} \frac{\Delta'}{\Delta}\left(it\right) = O_j(|t|^{-j}).
\end{align}

 For (\ref{k_bigger_0}), push the integral rightward to the line $\Re(s)  = \frac{A+1}{\xi} +\delta$ with $0 < \delta < 1$ chosen so that the contour has distance bounded away from any pole of the integrand.  On this line, the integral may be bounded in absolute value, and has appropriate size. In shifting the contour, we pass poles from $\Delta$ on the real axis  and boundedly many poles of $\frac{1}{1-\omega \Delta(s)}$ all lying within a bounded distance of the real axis, but each of these are smaller than any negative power of $T$  due to the factor of cosine in the denominator.   
 
 For (\ref{k_equal_0}), shift to $\Re(s) =0$ and apply the approximation to 
$\frac{\Delta'}{\Delta}(s)$.  When $m \neq n$ write 
\begin{equation}\left(\frac{m}{n}\right)^s =  e^{it \log 
\frac{m}{n}}\end{equation} and integrate  by parts  in $t$ several times.  
\end{proof}

As a first example, we calculate the normalizing weight in the probability measure $w_{T,\theta}$. In this calculation we use  symmetries $\Delta(s) = \frac{1}{\Delta(-s)}$, $\frac{\Delta'}{\Delta}(s) = \frac{\Delta'}{\Delta}(-s)$.
\begin{lemma}\label{weight_lemma}
 Recall the definition $r'(p) = \frac{r(p)}{1 + r(p)^2}.$  We have
 \begin{equation}
  \sum_{t \in T_\theta} \frac{|R( i t)|^2}{\cosh(\frac{T - t}{H})}\asymp_\xi (\log T)^{\frac{1}{4}} \prod_p (1 + r(p)^2) \prod_p \left(1 + \frac{r'(p)}{\sqrt{p}}\right)  \int_{t >20}  \frac{\log(\frac{t}{2\pi} ) dt}{\cosh\left(\frac{T-t}{H}\right)} .
 \end{equation}

\end{lemma}
\begin{proof}
 To within negligible error, the left hand side is
 \begin{equation}
 \sim \oint_{\Gamma_\epsilon} R(s)R(-s) \frac{\Delta'(s)}{\Delta(s) - e^{2i \theta}} \frac{ds}{\cos\left(\frac{ iT - s}{H}\right)},
 \end{equation}
 since the boundedly many poles that do not fall on the half-line get exponentially small weight.
The integral on the $\Re(s) = \frac{1}{2} + \epsilon$  is negligible as can be 
seen by using (\ref{k_bigger_0}) with \begin{equation}\left(\sum_n a_n\right)^2 
\leq T^{1-2\xi} \sum a_n^2.\end{equation}  Substituting $s \mapsto -s$ in the 
integral on the $\Re(s) =-\frac{1}{2}-\epsilon$ line, it becomes
\begin{equation}
 \oint_{\Re(s) =\frac{1}{2} + \epsilon} R(s) R(-s) \left(-\frac{\Delta'(s)}{\Delta(s)}\right)\frac{1}{1-e^{2i\theta}\Delta(s)} \frac{ds}{\cos\left(\frac{ iT + s}{H}\right)}.
\end{equation}
When $1/(1-\Delta(s)e^{2i\theta})$ is expanded in geometric series, only the constant term contributes.  Applying (\ref{k_equal_0}) we pick up diagonal terms with negligible error.  These give the stated integral times $\sum_{n <T^{1-2\xi}} a_n^2$, which was evaluated asymptotically in Lemma \ref{eval_an_square_lemma}.
\end{proof}

We abbreviate this normalizing weight $\rNW$.

\subsection{Main estimates}
For the remainder of the argument we employ a simple truncation to approximate $\zeta(\frac{1}{2} + s)$ by a Dirichlet polynomial.  Fix a smooth function $\phi: \bR\to [0,1]$, supported in $[-1,1]$, satisfying $\phi\equiv 1$ in a neighborhood of 0.  By Mellin inversion, one obtains that, uniformly in $\left\{s = \sigma + it: \frac{T}{2} \leq t \leq 2T, \; 0 \leq \sigma \leq 2\right\}$, 
\begin{equation}\label{zeta_truncation}
\forall\; \epsilon > 0,\; \forall\; A > 0, \qquad
 \zeta\left(\frac{1}{2} + s\right) = \sum_n \frac{1}{n^{\frac{1}{2} + s}} \phi\left(\frac{n}{T^{1 + \epsilon}}\right) + O_{A,\epsilon}(T^{-A}).
\end{equation}
In the integrals below we use this approximation also for $t \not \in \left[\frac{T}{2}, 2T\right]$.  The error in doing so is negligible due to the rapid decay of the factor from cosine, which in this range is smaller than any fixed negative power of $T$.  

\subsubsection{The  lower bound of Proposition \ref{main_prop}}
We bound

\begin{equation}
 \E_{w_{T,\theta}}\left[ \left|\zeta\left(\frac{1}{2} + it\right)\right|\right]  \geq \left|\E_{w_{T,\theta}} \left[\zeta\left(\frac{1}{2} -it\right)\frac{A_{\frac{1}{2}}\left(\frac{1}{2} + it\right)^2}{\left|A_{\frac{1}{2}}\left(\frac{1}{2} + it\right)\right|^2} \right]\right|.
\end{equation}

We may write
\begin{align}
 &\E_{w_{T,\theta}} \left[\zeta\left(\frac{1}{2} 
-it\right)\frac{A_{\frac{1}{2}}\left(\frac{1}{2} + 
it\right)^2}{\left|A_{\frac{1}{2}}\left(\frac{1}{2} + it\right)\right|^2} 
\right] \cdot \rNW\\ \notag &\sim \oint_{\Gamma_\epsilon} 
\zeta\left(\frac{1}{2} -s\right)  A_{\frac{1}{2}}\left(\frac{1}{2} +s\right)^2  
R^*(s)R^*(-s)\frac{\Delta'(s)}{\Delta(s) - e^{2i \theta}} 
\frac{ds}{\cos\left(\frac{ iT - s}{H}\right)}  .
\end{align}

On the line $\Re(s) = \frac{1}{2}+\epsilon$, write $\zeta(\frac{1}{2}-s) = \frac{\zeta(\frac{1}{2}+s)}{\Delta(s)}$  and expand $\frac{\Delta'(s)}{\Delta(s) - e^{2i \theta}}$ in geometric series.  Appealing to (\ref{k_bigger_0}) of Lemma \ref{delta_integral_lemma}, one easily checks that all but the first term in the geometric series expansion give negligible contribution.  The first term is
\begin{equation}
 \frac{1}{e^{2i \theta}} \oint_{\Re(s) = \frac{1}{2} +\epsilon} \zeta\left(\frac{1}{2} +s\right)  A_{\frac{1}{2}}\left(\frac{1}{2} +s\right)^2  R^*(s)R^*(-s)\left(-\frac{\Delta'}{\Delta}(s)\right) \frac{ds}{\cos\left(\frac{ iT - s}{H}\right)}.
\end{equation}
Approximate $\zeta$ with its Dirichlet series (\ref{zeta_truncation}) with negligible error, and apply (\ref{k_equal_0}) of Lemma \ref{delta_integral_lemma} to obtain diagonal terms plus an error which is negligible.  Associating variables $n$ to $\zeta$, $m_1, m_2$ to the two powers of $A_{\frac{1}{2}}$ and $\ell_1, \ell_2$ to $R^*(s)$ and $R^*(-s)$, the diagonal condition is $nm_1m_2\ell_1 = \ell_2$, and thus the diagonal term is given by (replace $\ell_2 := \frac{\ell_2}{\ell_1}$ and note that $\phi$ does not enter due to the range of summation)
\begin{equation}
 \frac{-1}{e^{2i\theta}}  \sum_{\ell_1} r(\ell_1)^2 \sum_{\substack{\ell_2 \leq \frac{N}{\ell_1}\\(\ell_2,\ell_1) = 1}} \frac{r(\ell_2)}{\sqrt{\ell_2}} \sum_{\substack{nm_1m_2 = \ell_2\\ m_1, m_2 < T^\xi}} d_{\frac{1}{2}}(m_1)d_{\frac{1}{2}}(m_2) \times \int_{t\geq 20} \frac{(\log \frac{t}{2\pi}+O(\frac{1}{t})) dt}{\cosh\left(\frac{t-T}{H}\right)}.
\end{equation}
The inner sum is bounded by $d(\ell_2)$, and so, after dividing by the normalizing weight, we find that this term contributes a quantity which is
\begin{equation}
 \ll (\log T)^{\frac{-1}{4}} \prod_p \left(1 + \frac{r'(p)}{\sqrt{p}}\right),
\end{equation}
which is an error term.

On the line $\Re(s) = -\frac{1}{2} -\epsilon$, exchange $s\mapsto -s$ to write the integral as
\begin{equation}
 \oint_{\Re(s) = \frac{1}{2}+\epsilon} \zeta\left(\frac{1}{2} + s\right)A_{\frac{1}{2}}\left(\frac{1}{2}-s\right)^2 R^*(s) R^*(-s) \left(-\frac{\Delta'}{\Delta}(s)\right) \frac{1}{1 - e^{2i\theta}\Delta(s)} \frac{ds}{\cos\left(\frac{iT + s}{H}\right)}.
\end{equation}
Expanding $\frac{1}{1-e^{2i\theta}\Delta(s)}$ in geometric series, only the constant term contributes.  In this term, replace $\zeta$ with its approximating Dirichlet polynomial and take only the diagonal terms from the resulting integral.  With the same variable conventions as before, and pulling out $g = (\ell_1, \ell_2)$, this yields
\begin{equation}
  \sum_g r(g)^2 \sum_{\substack{\ell_1, \ell_2 \leq \frac{N}{g}\\ (\ell_1,\ell_2) = (\ell_1\ell_2,g) = 1}} r(\ell_1)r(\ell_2) \sum_{\substack{m_1,m_2 \leq T^\xi\\ \ell_1 n=\ell_2 m_1m_2}} \frac{d_{\frac{1}{2}}(m_1)d_{\frac{1}{2}}(m_2)}{\sqrt{nm_1m_2}} \times \int_{t\geq 20} \frac{(\log \frac{t}{2\pi} +O(\frac{1}{t}))dt}{\cosh\left(\frac{t-T}{H}\right)} . 
\end{equation}
Since we seek only a lower bound, truncate to $\ell_1, \ell_2 < Z = \exp\left((\log T)^{\frac{2}{3}}\right)$.  Also, write $\ell_{11} = (\ell_1, m_1)$, $\ell_{12} = (\ell_1, m_2)$ and replace $m_1 := \frac{m_1}{\ell_{11}}$, $m_2:= \frac{m_2}{\ell_{12}}$.  Omitting the integral, the sum becomes
\begin{align}
& \sum_{\substack{\ell_1, \ell_2 < Z\\ (\ell_1, \ell_2) = 1}} \frac{r(\ell_1) 
r(\ell_2)}{\sqrt{\ell_1\ell_2}}\sum_{\ell_{11}\ell_{12} = 
\ell_1}\sum_{\substack{g \leq \frac{N}{\max(\ell_1,\ell_2)}\\ (g, 
\ell_1\ell_2)=1}} r(g)^2 \sum_{\ell_{11}m_1, \ell_{12}m_2 \leq T^\xi} 
\frac{d_{\frac{1}{2}}(\ell_{11}m_1)d_{\frac{1}{2}}(\ell_{12} m_2)}{m_1m_2}\\ 
\notag
&\gg_\xi \log T \prod_p \left(1 + \frac{2r'(p)}{\sqrt{p}}\right) \prod_p (1 + r(p)^2),
\end{align}
where we use (\ref{one_factor_lower_bound}) of Lemma \ref{divisor_extraction} to evaluate the sums over $m_1, m_2$, and use the concentration properties of $r$ to evaluate the resulting sums over $\ell_1, \ell_2$ and $g$.  Inserting the integral over $t$ and dividing by the normalizing weight, we arrive at the claimed main term.

\subsubsection{The upper bound of Proposition \ref{main_prop}}
The signed expectation is given by
\begin{align}
 &\left|\E_{w_{T,\theta}}\zeta\left(\frac{1}{2} + it\right) \right|\cdot \rNW
\\ \notag &=
 \left|\oint_{\Gamma_\epsilon}\zeta\left(\frac{1}{2}-s\right) A_{\frac{1}{2}}\left(\frac{1}{2} + s\right) A_{\frac{1}{2}}\left(\frac{1}{2}-s\right) R(s) R(-s)  \frac{\Delta'(s)}{\Delta(s)-e^{2i\theta}} \frac{ds}{\cos\left(\frac{ iT - s}{H}\right)}\right|.
\end{align}
The argument for estimating these integrals is exactly as in the lower bound, so we only describe the evaluation of the diagonal terms.  Furthermore, up to a constant of absolute value 1, each diagonal term is given by
\begin{equation}
 \sum_g r(g)^2 \sum_{\substack{\ell_1, \ell_2 \leq \frac{N}{g}\\ (\ell_1,\ell_2) = (\ell_1\ell_2,g) = 1}} r(\ell_1)r(\ell_2) \sum_{\substack{m_1,m_2 \leq T^\xi\\ \ell_1 m_1 n=\ell_2 m_2}} \frac{d_{\frac{1}{2}}(m_1)d_{\frac{1}{2}}(m_2)}{\sqrt{nm_1m_2}} \times \int_{t\geq 20} \frac{(\log \frac{t}{2\pi}+O(\frac{1}{t})) dt}{\cosh\left(\frac{t-T}{H}\right)}
\end{equation}
Write $\ell_{21} = (\ell_2, m_1)$, $\ell_{22} = (\ell_2, n)$, replace $m_1 := \frac{m_1}{\ell_{21}}$, $n:= \frac{n}{\ell_{22}}$ and appeal (\ref{two_divisor_upper_bound}) of  Lemma \ref{divisor_extraction} to estimate
\begin{align}
 &\sum_g r(g)^2 \sum_{\substack{\ell_1, \ell_2 \leq \frac{N}{g}\\ (\ell_1,\ell_2) = (\ell_1\ell_2,g) = 1}} \frac{r(\ell_1)r(\ell_2)}{\sqrt{\ell_1\ell_2}} \sum_{\ell_{21}\ell_{22} = \ell_2} \sum_{\ell_{21}m_1, \ell_1 m_1 n < T^\xi} \frac{d_{\frac{1}{2}}(\ell_{21}m_1 ) d_{\frac{1}{2}}(\ell_1 m_1 n )}{m_1 n}
 \\ \notag
 &\ll (\log T)^{\frac{3}{4}} \sum_g r(g)^2 \sum_{\substack{\ell_1, \ell_2 \leq 
\frac{N}{g}\\ (\ell_1,\ell_2) = (\ell_1\ell_2,g) = 
1}}\frac{r(\ell_1)r(\ell_2)d_{\frac{1}{2}}(\ell_1)d_{\frac{3}{2}}(\ell_2)}{\sqrt
{\ell_1\ell_2}} \\ \notag &\ll (\log T)^{\frac{3}{4}} \prod_p\left(1 + 
\frac{2r'(p)}{\sqrt{p}}\right) \prod_p (1 + r(p)^2).
\end{align}
Reinserting the integral and dividing by the normalizing weight, we arrive at the claimed upper bound.

\section{The argument of large central values of Dirichlet $L$-functions}

In what follows we restrict consideration to the $L$-functions associated to even $(\chi(-1)=1)$ primitive characters modulo large prime $q$.  In particular, these satisfy a common functional equation
\begin{equation}
 \Lambda(s,\chi) = \Gamma\left(\frac{s}{2}\right) \left(\frac{q}{\pi}\right)^{\frac{s}{2}} L(s,\chi) = \frac{\tau(\chi)}{\sqrt{q}}\Lambda(1-s,\overline{\chi}). 
\end{equation}We indicate uniform expectation over primitive even characters by 
$\E_{\chi \bmod q}^+$.

Let $N = q^\eta$, $0< \eta < \frac{1}{32}$, set $L = \sqrt{\log N\log\log N}$ and define resonating polynomial
\begin{equation}
 R(\chi) = \sum_{n \leq N} r(n) \chi(n).
\end{equation}
As in the case of $\zeta$, we supplement our  resonator with a short Dirichlet polynomial. 
For a small $\xi$, $0 < \xi < \frac{1}{16}-\frac{\eta}{2}$, let
\begin{equation}
A_{\frac{1}{2}}(\chi) = \sum_{n \leq q^{\xi}} \frac{d_{\frac{1}{2}}(n)\chi(n)}{\sqrt{n}}. 
\end{equation}
This, in a distributional sense, behaves somewhat like $L\left(\frac{1}{2}, 
\chi \right)^{\frac{1}{2}}$.

Define, for $f: \widehat{\zed/q\zed} \to \bC$ 
\begin{equation}\E_R[f] = \frac{\E_{\chi\bmod q}^+ \left[|R(\chi)|^2 
f(\chi)\right]}{ \E_{\chi \bmod q}^+ \left[|R(\chi)|^2\right]}.\end{equation}

The key estimates used in proving Theorem \ref{dirichlet_L_angle_theorem} are as follows.
\begin{theorem}[Estimates of Weyl Type]\label{dirichlet_moment_theorem}
 Let $\sL = \log q \prod_p \left(1 + \frac{r(p)}{\sqrt{p}(1 + r(p)^2) }\right)$.
 We have the following estimates for moments of $L(\frac{1}{2},\chi)$ (all implicit constants depend upon $\xi$ and $\eta$). 
 \begin{enumerate}
  \item[a.] The central moments satisfy
   \begin{align}
  \E_{R}\left[ \left|L\left(\frac{1}{2},\chi\right)\right|^2 
\left|A_{\frac{1}{2}}(\chi)\right|^4\right] &\gg \sL^4,\\ \notag
 \E_{R}\left[ \left|A_{\frac{1}{2}}(\chi)\right|^8\right],\;  \E_{R}\left[ \left|L\left(\frac{1}{2},\chi\right)\right|^4\right] &\ll \sL^4.
 \end{align}
  \item[b.]  The signed moments satisfy 
  \begin{align}    
\E_R\left[L\left(\frac{1}{2},\chi\right)\left|A_{\frac{1}{2}}
(\chi)\right|^6\right] &\ll (\log q)^{\frac{-1}{4}} \sL^4\\ 
\E_R\left[L\left(\frac{1}{2}, \chi\right)^2 
\left|A_{\frac{1}{2}}(\chi)\right|^4\right] &\ll (\log q)^{-1} \sL^4. \notag 
  \end{align}
  \item[c.] Twisting by higher phases, there is $B>0$ such that  
\begin{align}
 \forall\, m\geq 1, \qquad &  \E_R\left[e(2m\theta_\chi) 
L\left(\frac{1}{2},\chi\right)\left|A_{\frac{1}{2}}(\chi)\right|^6\right], \; 
\E_R\left[e(2m\theta_\chi) 
L\left(\frac{1}{2},\chi\right)^2\left|A_{\frac{1}{2}}(\chi)\right|^4\right]\\& 
\ll_{ \epsilon, B} (m+5)^B q^{-\frac{1}{8} + \eta + 2\xi+ \epsilon} \sL^4 . 
\notag
\end{align}
 
 \end{enumerate}

\end{theorem}

Domination by the unsigned moments of the same moments twisted by phases $e(m\theta_\chi), 0<|m| < M$   suggests angular equidistribution of $R$-typical values of $L(\frac{1}{2},\chi)$ at a scale of $\asymp \frac{1}{M}$, e.g. as in Weyl's criterion for equidistribution. 
We make this intuition precise by appealing to the following quantitative equidistribution result, which we use to prove Theorem \ref{dirichlet_L_angle_theorem}.

Let $\bT = \bR/\zed$ be the usual torus, with distance $\|\cdot\| 
=\|\cdot\|_{\bR/\zed}$.    
\begin{theorem}[Minorant theorem]\label{minorant_theorem}
Let $A_1, A_2 > 0$ be constants, and let $\delta \in \left(0, 
\frac{1}{4}\right)$, $\theta \in \bR/\zed$ be parameters.  There exist even and 
odd trigonometric polynomials $F^e_{\delta}(t), F^o_{\delta}(t)$ and constants 
$c_0 = \int_{\bT} F^e_{\delta}(t) dt, c_1, c_2, c_3, c_4 > 0$ 
depending at most upon $A_1, A_2$ and $\delta$ and satisfying the following 
properties. Below $*$ represents either $e$ or $o$.
\begin{enumerate}
 \item $\left\|F^*_\delta\right\|_{L^\infty}, \left\|\widehat{F}^*_{\delta}\right\|_{\ell^1} \ll 1 $ and $ \left\|\widehat{F}^*_{\delta}\right\|_{\ell^\infty} \ll \delta $

 \item $
        \deg F^*_{\delta}\ll \frac{1}{\delta (-\log \delta)^3}
       $
 \item $
        c_1 \asymp \delta, \quad c_2 \asymp \delta^{-\frac{1}{2}}, \quad 
c_3 \asymp \delta 
       $
 \item
 $ 
  c_0 - A_1 c_1 - A_2 c_3\gg \delta.
 $
\end{enumerate}
Furthermore, the function in polar coordinates $M_{\delta,\theta} : \bR_{\geq 0} 
\times \bT \to \bR$,
\begin{equation}
 M_{\delta,\theta}(r,t) = F^e_{\delta}(t-\theta)r^2 - c_1 + c_2 r F^o_{\delta}(t-\theta) - c_3 r^4
\end{equation}
satisfies
\begin{equation}
 \one\left( (r,t): M_{\delta,\theta}(r,t) \geq 0\right) \leq \one\left(r > c_4\delta^{\frac{3}{2}}\right) \cdot \one( \|t-\theta\|_{\bT} < \delta).
\end{equation}

\end{theorem}
\begin{proof}
  Let \begin{equation}\sigma(x) = e^{-\frac{1}{1-x^2}} \one( |x| < 
1)\end{equation} be the standard bump function with support in $[-1,1]$.  The 
Fourier transform of $\sigma$, 
 \begin{equation}
 \widehat{\sigma}(\xi) = \int_{-1}^1 \sigma(x) e^{2\pi i x \xi} dx, \qquad \xi \in \bR
 \end{equation}
 satisfies the decay property
 \begin{equation}
  \left|\widehat{\sigma}(\xi)\right| \ll \frac{1}{|\xi|^{\frac{3}{4}}} e^{-\sqrt{2\pi|\xi|}}
 \end{equation}
as $|\xi|\to\infty$.  
 Write $\sigma_\delta(x) = \sigma\left(\frac{x}{\delta}\right)$ and note that $\widehat{\sigma}_\delta(\xi) = \delta \widehat{\sigma}(\delta \xi)$. Treat $\sigma_\delta$ as a function on $\bT$ and define the even and odd parts of $\sigma_\delta$ by
 \begin{equation}
  \sigma_\delta^e(x) = \frac{\sigma_\delta(x) + \sigma_\delta\left(x + \frac{1}{2}\right)}{2}, \qquad \sigma_\delta^o(x) = \frac{\sigma_\delta(x) - \sigma_\delta\left(x + \frac{1}{2}\right)}{2}.
 \end{equation}

 Let $\epsilon = \frac{\widehat{\sigma}(0)}{2}$. 
 For a sufficiently large constant $C>0$  let $N = \frac{C}{\delta (-\log \delta)^3}$.  Let $T_N$ be the operator on $L^2(\bT)$ which truncates the Fourier series at degree $N$. We define $F^e_{\delta} = -\epsilon \delta + T_N(\sigma_\delta^e),$ and $F^o_{\delta} = T_N(\sigma_\delta^o)$. 
 Evidently 
 \begin{equation}
  c_0 = \int_{\bT} F^e_\delta(t)dt = \frac{\delta}{2}\widehat{\sigma}(0).
 \end{equation}
We choose
\begin{equation}
c_1 = \frac{c_0}{3\max(A_1, 1)},  \qquad c_2 = 
\sqrt{\frac{1}{2c_3}}, \qquad c_3 = \frac{c_0}{3\max(A_2,1)},
\end{equation}
which guarantees
\begin{equation}
 c_0 - A_1 c_1 - A_2 c_3 \geq \frac{c_0}{3} \gg \delta.
\end{equation}

 We have
 \begin{equation}
  \left\|\widehat{F}_{\delta}^*\right\|_{\ell^\infty} \leq \|\sigma_\delta\|_{L^1(\bT)} \ll \delta
 \end{equation}
while from the decay properties of $\widehat{\sigma}$ one readily checks by 
bounding the Fourier series in absolute value that (the inequality fixes $C$, 
the claimed bound is not tight as $\delta \downarrow 0$)
\begin{equation}
 \left\|\widehat{F}_{\delta}^*\right\|_{\ell^1} \ll 1, \qquad \|\epsilon \delta +F_{\delta}^e - \sigma_\delta^e\|_{L^\infty(\bT)}, \;\|F_{\delta}^o - \sigma_\delta^o\|_{L^\infty(\bT)} \leq \min\left(\frac{\epsilon \delta}{2}, \frac{c_1}{2c_2^2}\right).
\end{equation}  

To check the minorant property of $M_{\delta, \theta}$ we may assume $\theta = 0$.  For $\min\left(\|t\|_{\bT}, \left\|t-\frac{1}{2}\right\|_{\bT}\right) > \delta$, $(r,t)$ is outside the support of both $\sigma_\delta^e$ and $\sigma_\delta^o$, so that
\begin{equation}
 r^2 F_\delta^e(t) + c_2 r F_\delta^o(t) \leq -\epsilon \delta r^2 + \min\left(\frac{\epsilon\delta}{2}, \frac{c_1}{2c_2^2}\right)(c_2 r + r^2) \leq c_1 + c_3 r^4,  
\end{equation}
as may be checked separately for $r \leq c_2$ and $r > c_2$.  Thus, for this range of $t$  the minorant property holds.

Next consider $t = s + \frac{1}{2}$ with $\|s\|_{\bT} \leq \delta$.  Here, again,
\begin{align}
 r^2 F_\delta^e(t) + c_2 r F_\delta^o(t) &\leq \frac{r^2 - c_2 
r}{2}\sigma_\delta(s) -\epsilon \delta r^2 + \min\left(\frac{\epsilon 
\delta}{2}, \frac{c_1}{2c_2^2}\right) (c_2 r + r^2) \\& \leq c_1 + c_3 r^4 
\notag
\end{align}
by checking separately for $r \leq c_2$ and $r > c_2$.

To confirm the final minorant property, note that for $\|t\|_\bT < \delta$ and $r < 1$ we have
\begin{align}
 r^2 F_\delta^e(t) + c_2 r F_\delta^o(t) &\leq \frac{r^2 + c_2 
r}{2}\sigma_\delta(s) -\epsilon \delta r^2 +\min\left(\frac{\epsilon \delta}{2}, 
\frac{c_1}{2c_2^2}\right) (c_2 r + r^2) \ll c_2 r 
\end{align}
so that $r^2 F_\delta^e(t) + c_2 r F_\delta^o(t) \leq c_1$ once $r \ll \frac{c_1}{c_2} \ll \delta^{\frac{3}{2}}$.

\end{proof}

Combining Theorems \ref{dirichlet_moment_theorem} and \ref{minorant_theorem} we prove Theorem \ref{dirichlet_L_angle_theorem}.
\begin{proof}[Proof of Theorem \ref{dirichlet_L_angle_theorem}]
Apply Theorem \ref{minorant_theorem} with  
\begin{equation}A_1 =\frac{\E_R\left[\left|A_{\frac{1}{2}}(\chi)\right|^8 
\right]}{\E_R\left[\left|A_{\frac{1}{2}}(\chi)\right|^4 
\left|L\left(\frac{1}{2}, \chi\right) \right|^2 \right]}, \qquad A_2 = 
\frac{\E_R\left[ \left|L\left(\frac{1}{2}, \chi\right) 
\right|^4\right]}{\E_R\left[\left|A_{\frac{1}{2}}(\chi)\right|^4 
\left|L\left(\frac{1}{2}, \chi\right) \right|^2\right]}\end{equation} 
to obtain minorant $M_{\delta, \theta}$.  Theorem \ref{minorant_theorem} 
guarantees that there exists  constant $C(\delta) \ll \delta^{-6}$ such that 
\begin{equation}
\sF_{\delta,\theta}(r,t) = C(\delta) r^4 \one(\|t-\theta\|_{\bT} \leq \delta) \geq  M_{\delta, \theta}(r,t).
\end{equation}
Thus
\begin{align}
 C(\delta) \E_R\left[\left|L\left(\frac{1}{2},\chi\right)\right|^4 \one\left(\|\theta_\chi - \theta\|_{\bT} \leq \delta\right) \right] \geq \E_R\left[\left|A_{\frac{1}{2}}(\chi) \right|^8 M_{\delta,\theta} \left(\frac{\left|L\left(\frac{1}{2}, \chi\right)\right|}{\left|A_{\frac{1}{2}}(\chi) \right|^2} , \theta_\chi\right) \right].
\end{align}
Write $\sigma_\delta(t) = \sum_n b_n e(nt)$ in Fourier series.  We expand 
\begin{align}\notag
 \left|A_{\frac{1}{2}}(\chi) \right|^8 M_{\delta,\theta} \left(\frac{\left|L\left(\frac{1}{2}, \chi\right)\right|}{\left|A_{\frac{1}{2}}(\chi) \right|^2} , \theta_\chi\right)
 &=
 c_0  \left|A_{\frac{1}{2}}(\chi) \right|^4\left|L\left(\frac{1}{2}, \chi\right)\right|^2\\ + &2\Re\sum_{0 \leq n \leq \frac{N-2}{2}} b_{2n+2} e((2n+2)\theta - 2n\theta_\chi))  \left|A_{\frac{1}{2}}(\chi) \right|^4L\left(\frac{1}{2}, \chi\right)^2
 \\\notag + &2c_2 \Re\sum_{0 \leq n \leq \frac{N-1}{2}} b_{2n+1} e((2n+1)\theta 
- 2n\theta_\chi))  \left|A_{\frac{1}{2}}(\chi) \right|^6L\left(\frac{1}{2}, 
\chi\right)\\ \notag
  -& c_1  \left|A_{\frac{1}{2}}(\chi) \right|^8- c_3 \left|L\left(\frac{1}{2},\chi \right) \right|^4.
\end{align}
By the estimates for unsigned moments in Theorem \ref{dirichlet_moment_theorem}, the first and last two terms combined have expectation
$
 \gg \delta \sL^4,
$
while, by the estimates for signed moments, the $n = 0$ terms in the two Fourier expansions have size
\begin{equation}
 \ll \delta c_2 (\log q)^{-\frac{1}{4}} \sL^4.
\end{equation}
Since $c_2 \ll \delta^{-\frac{1}{2}} = o\left((\log q)^{\frac{1}{4}}\right)$ these make negligible contribution.  By the estimates for twisted moments, the remaining terms of the Fourier expansion are $\ll q^{-\vartheta}$ for some $\vartheta > 0$, and thus are also negligible.
Since
\begin{equation}
 \log \sL \gtrsim \sqrt{\frac{\eta \log q}{\log \log q}}
\end{equation}
with the condition $\eta < \frac{1}{32}$, we obtain the result.

\end{proof}

\subsection{Estimations of Weyl type}
 For convenience we restrict attention to the central values of the $\frac{q-3}{2}$ even primitive characters. Note that the orthogonality relation for these characters is given by
\begin{equation}
\forall n_1, n_2 \not \equiv 0 \bmod q, \qquad 
 \E_{\chi \bmod q}^+[\chi(n_1)\overline{\chi}(n_2)] = \left\{ \begin{array}{lll} 1 && n_1 \equiv \pm n_2  \bmod q\\ \frac{-2}{q-3} && \text{otherwise}\end{array}.\right.
\end{equation}
The negative off-diagonal term results from removing the principal character.

Our proofs use various formulas  expressing the central values $L(\frac{1}{2},\chi)$ and their powers as truncated Dirichlet series. 
Fix a smooth function $\phi:\bR_{\geq 0} \to [0,1]$ with support in $[0,1]$ and such that $\phi \equiv  1$ on a neighborhood of $0$.  Mellin inversion gives, for all $\epsilon > 0$ and all $A>0$,
\begin{equation}\label{dirichlet_approx}
 L\left(\frac{1}{2},\chi\right) = \sum_n \frac{\chi(n)}{\sqrt{n}} \phi\left(\frac{n}{q^{1 + \epsilon}}\right) + O_{A,\epsilon}\left(q^{-A}\right).
\end{equation}

The approximate functional equation gives representations of the central values as shorter sums (\cite{IK04}, p. 98).   One has
\begin{equation}\label{AFE_signed_square}
 L\left(\frac{1}{2},\chi\right)^2 = \sum_n \frac{\chi(n)d(n)}{\sqrt{n}} V_1\left(\frac{\pi n}{q}\right) + \frac{\tau(\chi)^2}{q} \sum_n \frac{\overline{\chi}(n) d(n)}{\sqrt{n}} V_1^*\left(\frac{\pi n}{q}\right),
\end{equation}
and
\begin{align}\label{AFE_unsigned_square}
 \left|L\left(\frac{1}{2},\chi\right)\right|^2 &= 2 \sum_n \frac{d_\chi(n)}{\sqrt{n}}V_2\left(\frac{\pi n}{q}\right), \qquad d_\chi(n) = \sum_{n_1n_2 = n} \chi(n_1)\overline{\chi(n_2)}.
\end{align}
Here $V_i, V_i^*$ are some Schwartz-class functions on $\bR$ satisfying $V_i(0) = V_i^*(0) = 1$.  It will be convenient to assume, as we may, that $V_i$ and $V_i^*$ are non-negative.

Recall the definition of the expectation \begin{equation}\E_R[f] = 
\frac{\E_{\chi\bmod q}^+ \left[|R(\chi)|^2 f(\chi)\right]}{ \E_{\chi \bmod q}^+ 
\left[|R(\chi)|^2\right]}.\end{equation} We write $\rNW$ for the normalizing 
weight in the denominator. 
\begin{lemma}
 The normalizing weight $\rNW$ satisfies
 \begin{equation}
  \rNW \sim \prod_p (1 + r(p)^2).
 \end{equation}

\end{lemma}
\begin{proof}
By orthogonality,
\begin{align}
    \rNW &= \sum_{n \leq q^{\eta}} r(n)^2 - \frac{2}{q-3} \sum_{\substack{n_1, 
n_2 \leq q^{\eta }\\ n_1 \neq n_2}}r(n_1)r(n_2)\\ \notag &=  \left(1 + 
O\left(q^{\eta - 1}\right)\right)  \sum_{n \leq q^{\eta}} r(n)^2  
\end{align}    
by applying Cauchy-Schwarz.  The sum is asympotitic to the product by using the tail bound in Lemma \ref{resonance_function_lemma}.
\end{proof}
\subsubsection{Estimation of central moments}\label{dirichlet_central_moments_section}
We prove part a  of Theorem \ref{dirichlet_moment_theorem}.

To estimate the second central moment from part a, in
\begin{align}
 \rNW \sM_2 =\E_{\chi \bmod q}^+ \left[ |R(\chi)|^2 \left|A_{\frac{1}{2}}(\chi)\right|^4 \left|L\left(\frac{1}{2}, \chi\right)\right|^2\right] 
\end{align}
expand $\left|L\left(\frac{1}{2}, \chi\right)\right|^2$ via the approximate functional equation (\ref{AFE_unsigned_square}), and use orthogonality to obtain a diagonal term
\begin{align}
2 \sum_g r(g)^2 \sum_{\substack{\ell_1, \ell_2 \leq \frac{N}{g}\\ 
(\ell_1,\ell_2)=(\ell_1\ell_2,g)=1}} r(\ell_1) r(\ell_2) \sum_{n_1, n_2 
< q^{2\xi}} \frac{d_{\frac{1}{2}, 2, q^\xi}(n_1)d_{\frac{1}{2}, 2, 
q^\xi}(n_2)}{\sqrt{n_1n_2}} \sum_{\substack{m_1,m_2\\  \ell_1n_1m_1 = 
\ell_2n_2m_2}}\frac{V_2\left(\frac{\pi m_1m_2}{q}\right)}{\sqrt{m_1m_2}}
\end{align}
plus an error term which is, for any $\epsilon > 0$,
\begin{equation} 
  \ll q^{-1}\sum_{\substack{\ell_1, \ell_2 \leq N}} r(\ell_1) r(\ell_2) 
\sum_{n_1, n_2 < q^{2\xi}} \frac{1}{\sqrt{n_1n_2}} 
\sum_{\substack{m_1,m_2}}\frac{V_2\left(\frac{\pi 
m_1m_2}{q}\right)}{\sqrt{m_1m_2}} \ll_\epsilon \rNW q^{\eta +2\xi - \frac{1}{2} 
+ \epsilon} = o(\rNW).
\end{equation}
Discarding the error, and restricting to $m_1n_1, m_2n_2 < q^\xi$ in the main term, thus removing the restricted divisor functions, we obtain the lower bound
\begin{equation}
\sM_2 +o(1) \gg  \frac{1}{\rNW}\sum_{g} r(g)^2 \sum_{\substack{\ell_1, \ell_2 \leq \frac{N}{g}\\ (\ell_1,\ell_2)=(\ell_1\ell_2,g)=1}} \frac{r(\ell_1) r(\ell_2)}{\sqrt{\ell_1\ell_2}} \sum_{m \leq q^\xi} \frac{d(\ell_1m)d(\ell_2 m)}{m}.
\end{equation}
Restricting to $\ell_1, \ell_2 < Z = \exp\left((\log q)^{\frac{2}{3}}\right)$ and applying 
 (\ref{divisor_lower_bound}) of Lemma \ref{divisor_extraction}, the inner sum is $\gg d(\ell_1)d(\ell_2) (\log q)^4$, so that we obtain the lower bound (we appeal to concentration properties of $r$ as before)
\begin{equation}
 \sM_2 +o(1)\gg \frac{(\log q)^4}{\rNW} \sum_{g} r(g)^2 \sum_{\substack{\ell_1, \ell_2 \leq \min\left(Z,  \frac{N}{g}\right)\\ (\ell_1,\ell_2)=(\ell_1\ell_2,g)=1}} \frac{r(\ell_1) r(\ell_2)d(\ell_1)d(\ell_2)}{\sqrt{\ell_1\ell_2}} \gg  \sL^4.
\end{equation}

Now we prove the upper bounds.
Treating the off-diagonal terms as above, we bound (use $d(\ell n)\leq d(\ell)d(n)$)
\begin{align}\notag
 \E_{\chi \bmod q}^+ \left[ |R(\chi)|^2 \left|A_{\frac{1}{2}}(\chi)\right|^8 
\right] &\ll \sum_{g} r(g)^2 \sum_{\substack{\ell_1, \ell_2 \leq \frac{N}{g}\\ 
(\ell_1,\ell_2)=(\ell_1\ell_2,g)=1}}r(\ell_1) r(\ell_2) \sum_{\substack{n_1, n_2 
\leq q^{4\xi}\\ \ell_1 n_1 = \ell_2n_2}} \frac{d(n_1)d(n_2)}{\sqrt{n_1n_2}}\\& 
 \ll (\log q)^4 \sum_{g} r(g)^2 \sum_{\substack{\ell_1, \ell_2 \leq 
\frac{N}{g}\\ (\ell_1,\ell_2)=(\ell_1\ell_2,g)=1}} \frac{r(\ell_1) 
r(\ell_2)d(\ell_1)d(\ell_2)}{\sqrt{\ell_1\ell_2}} \\& \notag \ll \rNW \sL^4.
\end{align}

To bound the unsigned fourth moment, we insert the following consequence of Theorem \ref{twisted_fourth_moment_theorem}.
\begin{cor}
Let $q$ be a large prime, let $\eta < \frac{1}{32}$ and let $\ell_1, \ell_2 < q^\eta$ be square-free, satisfying $(\ell_1, \ell_2) = 1$.  We have the upper bound
 \begin{equation}
  \E_{\chi\bmod q}^+ \left[ \chi(\ell_1)\overline{\chi}(\ell_2) 
\left|L\left(\frac{1}{2} + \alpha, \chi \right)\right|^4\right] \ll  
\frac{d(\ell_1) d(\ell_2)}{\sqrt{\ell_1\ell_2}} (\log q)^4.
 \end{equation}
\end{cor}

\begin{proof}
This follows from Theorem \ref{twisted_fourth_moment_theorem} on noting that
 if $D_k = P\left(\frac{\partial}{\partial \alpha}, \frac{\partial}{\partial \beta}, \frac{\partial}{\partial \gamma}, \frac{\partial}{\partial \gamma}\right)$ is an order $k$ differential operator, then 
 \begin{equation}D_k \log \tau_{\alpha, \beta, \gamma, 
\delta}(\ell)\Big|_{\alpha = \beta = \gamma = \delta = 0} \ll_P  (\log 
\ell)^k.\end{equation}
 \end{proof}

 We obtain immediately
 \begin{align}\notag
  \E_{\chi\bmod q}^+ \left[|R(\chi)|^2 \left|L\left(\frac{1}{2},\chi\right)\right|^4\right]& \ll  (\log q)^4 \sum_g r(g)^2 \sum_{\substack{\ell_1, \ell_2 \leq \frac{N}{g}\\
 (\ell_1,\ell_2)=(g, \ell_1\ell_2)=1}} \frac{r(\ell_1)r(\ell_2)d(\ell_1)d(\ell_2)}{\sqrt{\ell_1\ell_2}}\\& \ll \rNW \sL^4.
 \end{align}

\subsubsection{Estimation of signed moments twisted by phases}
Throughout this section we describe only the second moment.  The first can be handled similarly, the essential difference being that we replace the approximate functional equation for the second moment with the longer approximation (\ref{dirichlet_approx}) for the first.  

Write, for $m \in \zed$, $e\left(2m\theta_\chi\right) = \left(\frac{\tau(\chi)}{\sqrt{q}}\right)^m.$ Open the resonator and approximate functional equation as
\begin{align}\notag 
\forall m \geq 0, \quad \sM_{2,m} &= 
\E_R\left[e(2m\theta_\chi)\left|A_{\frac{1}{2}}(\chi)\right|^4 
L\left(\frac{1}{2},\chi\right)^2\right] \\&=\frac{1}{\rNW}\sum_{\ell_1, \ell_2 < 
q^{\eta}} r(\ell_1)r(\ell_2) \sum_{n_1, n_2 \leq q^{2\xi}}\frac{d_{\frac{1}{2}, 
2, q^\xi}(n_1)d_{\frac{1}{2}, 2, q^\xi}(n_2)}{\sqrt{n_1n_2}}\\& \notag \qquad 
\times\E_{\chi \bmod 
q}^+\Biggl[\chi(\ell_1n_1)\overline{\chi}(\ell_2n_2)\left(\frac{\tau(\chi)}{
\sqrt{q}}\right)^m \\&\notag \qquad\times \left( \sum_n 
\frac{\chi(n)d(n)}{\sqrt{n}} V_1\left(\frac{\pi n}{q}\right) + 
\frac{\tau(\chi)^2}{q} \sum_n \frac{\overline{\chi}(n) d(n)}{\sqrt{n}} 
V_1^*\left(\frac{\pi n}{q}\right)\right)\Biggr].
\end{align}
Introduce  for $(n,q)=1$  and $m \geq 1$ the hyper-Kloosterman sum
 \begin{equation}
  \rKL_m(n,q) = \sum_{\substack{n_1,...,n_m \bmod q\\ n_1...n_m \equiv n \bmod q}} e\left(\frac{n_1 + ... + n_m}{q}\right).
 \end{equation}
 We set $\rKL_0(n,q) = \one\{n \equiv 1 \bmod q\}$.
\begin{lemma}
 Let $a, b \not \equiv 0 \bmod q$.  For each $m \geq 0$ we have
 \begin{equation}
  \E_{\chi \bmod q}^+ \left[\chi(a)\overline{\chi}(b) \tau(\chi)^m\right] = 
\left(1 + \frac{2}{q-3}\right) \left(\rKL_m(\overline{a}b, q) + 
\rKL_m(-\overline{a}b, q)\right) + (-1)^{m+1} \frac{2 }{q-3}.
 \end{equation}
\end{lemma}
\begin{proof}
 This follows directly on expanding the Gauss sum and applying the orthogonality relation. The term $(-1)^{m+1} \frac{2}{q-3}$ results from removing the principal character.
\end{proof}

Applying this lemma in the first of the two sums resulting from the approximate functional equation gives
\begin{align}
&\sM_{2,m}^1 = \frac{1}{\rNW}  \sum_{\ell_1, \ell_2 < q^{\eta }} 
r(\ell_1)r(\ell_2) \sum_{n_1, n_2 < q^{2\xi}} 
\frac{d_{\frac{1}{2},2,q^\xi}(n_1)d_{\frac{1}{2},2,q^\xi}(n_2)}{\sqrt{n_1n_2}} 
\times\\ \notag
&\times \Biggl( \frac{q-1}{q-3} \sum_{\varepsilon = \pm 1}\sum_{(n, q)=1} 
\frac{\rKL_m (\varepsilon \overline{n \ell_1 n_1} \ell_2n_2, 
q)d(n)}{q^{\frac{m}{2}}\sqrt{n}} V_1\left(\frac{\pi n}{q}\right)\\& \notag 
\qquad \qquad + \frac{(-1)^{m+1} 2}{(q-3)q^{\frac{m}{2}}} \sum_{(n, q)=1} 
\frac{d(n)V_1\left(\frac{\pi n}{q}\right)}{\sqrt{n}}\Biggr).
\end{align}
while the second term gives
\begin{align}
&\sM_{2,m}^2 = \frac{1}{\rNW}  \sum_{\ell_1, \ell_2 < q^{\eta }} 
r(\ell_1)r(\ell_2) \sum_{n_1, n_2 < q^{2\xi}} 
\frac{d_{\frac{1}{2},2,q^\xi}(n_1)d_{\frac{1}{2},2,q^\xi}(n_2)}{\sqrt{n_1n_2}} 
\times\\ \notag
&\times \Biggl( \frac{q-1}{q-3} \sum_{\varepsilon = \pm 1}\sum_{(n, q)=1} 
\frac{\rKL_{m+2} (\varepsilon n\overline{ \ell_1 n_1} \ell_2 n_2, 
q)d(n)}{q^{\frac{m+2}{2}}\sqrt{n}} V_1^*\left(\frac{\pi n}{q}\right) \\& 
\notag \qquad\qquad + \frac{(-1)^{m+3} 2}{(q-3)q^{\frac{m+2}{2}}} \sum_{(n, 
q)=1} \frac{d(n)V_1^*\left(\frac{\pi n}{q}\right)}{\sqrt{n}}\Biggr).
\end{align}
The terms not involving hyper-Kloosterman sums are negligible.

The first sum in the case $m=0$ must be handled separately because it has no oscillatory term.  As the sum consists of positive terms, we bound $d_{\frac{1}{2}, 2, q^\xi}(n_i) \leq 1$ to obtain (cancel the factor of $\ell_1$ from $m_2$)
\begin{align}
 \sM_{2,0}^1 &\ll \frac{1}{\rNW} \sum_g r(g)^2 \sum_{\substack{\ell_1, \ell_2 
\leq \frac{N}{g}\\ (\ell_1,\ell_2) = (\ell_1\ell_2,g)=1}} r(\ell_1)r(\ell_2) 
\sum_{\substack{m_1, m_2 < q^{2\xi}, n \geq 1\\ \ell_1m_1 n = \ell_2 m_2}} 
\frac{d(n)V_1\left(\frac{\pi n}{q}\right)}{\sqrt{m_1m_2n}}\\ \notag
 &\ll \frac{1}{\rNW} \sum_g r(g)^2 \sum_{\substack{\ell_1, \ell_2 \leq \frac{N}{g}\\ (\ell_1,\ell_2) = (\ell_1\ell_2,g)=1}} \frac{r(\ell_1)r(\ell_2)}{\sqrt{\ell_1\ell_2}} \sum_{m \leq q^{2\xi}} \frac{d_3(\ell_2 m)}{m} \ll\frac{1}{\log q} \sL^4.
\end{align}

The second term in case $m=0$ and both terms for $m \geq 1$ contain non-trivial hyper-Kloosterman sums. These sums are trace functions of the type studied first by Katz \cite{K88} and further developed by Fouvry-Kowalski-Michel \cite{FKM14}.  In particular, when the Kloosterman sum is non-trivial,
\begin{equation}
 n \mapsto  \frac{\rKL_{m} (\pm \overline{ n\ell_1} \ell_2, q)}{q^{\frac{m-1}{2}}}, \qquad  n \mapsto  \frac{\rKL_{m+2} (\pm n\overline{ \ell_1} \ell_2, q)}{q^{\frac{m+1}{2}}}
\end{equation}
have conductors $m+3$ and $m+5$ respectively.
By making a dyadic partition of unity, it follows from \cite{FKM14} Theorem 1.15 that for any $\vartheta < \frac{1}{8}$, and for sufficiently large $B$
\begin{align}
 \sum_{(n,q)=1} \frac{\rKL_m (\pm \overline{n \ell_1} \ell_2, q)d(n)}{\sqrt{n} 
q^{\frac{m}{2}}} V_1\left(\frac{\pi n}{q}\right) &\ll_{\vartheta} (m+3)^B 
q^{-\vartheta}\\ \notag 
  \sum_{(n, q)=1} \frac{\rKL_{m+2} (\pm  n \overline{\ell_1}\ell_2, q)d(n)}{\sqrt{n} q^{\frac{m+2}{2}}} V_1^*\left(\frac{\pi n}{q}\right) &\ll_{ \vartheta} (m+5)^B q^{-\vartheta}.
\end{align}

Inserting these bounds together with  $r(\ell_1)r(\ell_2)\frac{d_{\frac{1}{2},2,q^\xi}(n_1)d_{\frac{1}{2},2,q^\xi}(n_2)}{\sqrt{n_1n_2}}< q^{\eta + 2\xi}\frac{r(\ell_1)r(\ell_2)}{\sqrt{\ell_1\ell_2}n_1n_2}$ above gives for sufficiently large $B$, [recall $r'(p) = \frac{r(p)}{1 + r(p)^2}$] 
\begin{align}
 \sM_{2, 0}^2, \; \sM_{2, m} &\ll_{ \vartheta } \frac{(m+5)^B q^{-\vartheta + \eta + 2\xi}}{\rNW}  \sum_{\ell_1, \ell_2 <N} \frac{r(\ell_1) r(\ell_2)}{\sqrt{\ell_1\ell_2}} \sum_{m_1, m_2 < q^{\xi}} \frac{1}{m_1m_2}, \qquad m \geq 1
\\& \notag  \ll_{\vartheta} (m+5)^B q^{-\vartheta + \eta + 2\xi} (\log 
q)^{2}\prod_p\left(1 + \frac{r(p)}{\sqrt{p}}\right).
 \end{align}
and after dividing by the normalizing weight, this gives the bound  claimed in part c. of Theorem \ref{dirichlet_moment_theorem}.

\end{document}